\documentclass[11pt]{amsart}

 \pdfoutput=1 % forces arXiv to use pdflatex
 \usepackage{amsmath, amssymb, amsthm, enumerate, bbm}
 \usepackage{graphicx,xcolor,soul}
 \usepackage[colorlinks, linkcolor=blue, citecolor=magenta, hypertexnames=false,backref]{hyperref} 
\usepackage{tikz} % (commutative) diagrams
\usetikzlibrary {decorations.pathmorphing,arrows.meta}

 \usepackage[backrefs]{amsrefs}

\def\spine{1.1in}
\usepackage[heightrounded, twoside, top=1in, bottom=1.3in, inner=\spine, outer=\spine]{geometry}

\def\epsilon{\varepsilon}

\def\restr{\scalebox{1.2}[1.3]{$\llcorner$}}

\theoremstyle{plain} % definition
\newtheorem{theorem}{Theorem}[section]
\newtheorem{lemma}[theorem]{Lemma}
\newtheorem{proposition}[theorem]{Proposition} 
\newtheorem{corollary}[theorem]{Corollary}

\newtheorem{definition}[theorem]{Definition}
\newtheorem{example}[theorem]{Example}

\numberwithin{equation}{section}

\title{Positive definiteness and the Stolarsky invariance principle}

\author{Dmitriy Bilyk}
\address{School of Mathematics, University of Minnesota, Minneapolis, MN 55408, USA.}
\email{dbilyk@math.umn.edu}

\author{Ryan W. Matzke} 
\address{School of Mathematics, University of Minnesota, Minneapolis, MN 55408, USA.}
\email{matzk053@umn.edu}

\author{Oleksandr Vlasiuk}
\address{Department of Mathematics, Florida State University, Tallahassee, FL 32306, USA.}
\email{ovlasiuk@fsu.edu}

\begin{document}

\maketitle

\begin{abstract}{In this  paper we elaborate on the  interplay between energy optimization, positive definiteness, and discrepancy.  In particular, assuming the existence of a $K$-invariant measure $\mu$ with full support, we show that conditional positive definiteness of a kernel $K$ is equivalent to a long list of other properties: including, among others, convexity of the energy functional, inequalities for mixed energies, and the fact that $\mu$ minimizes the energy integral in various senses.  In addition, we prove a very general form of the Stolarsky Invariance Principle on compact spaces, which connects energy minimization and discrepancy    and extends several previously known versions.}\end{abstract}

\tableofcontents

%%%%%%%%%%%%%%%%%%%%%%%%%%%%%%%%%%
\section{Introduction}

\subsection{Energy Minimization}

Let $\Omega$ be  a compact metric  space and let  the kernel  $K: \Omega \times \Omega \rightarrow \mathbb{R}$ be   continuous and symmetric, i.e. $K(x,y) = K(y,x)$  for all $x,y \in \Omega$. %(\ov{an overlap with notation section — should this be moved there?}) 
We denote by $\mathcal{M}(\Omega)$ the set of finite regular signed Borel measures on $\Omega$, and by $\mathbb{P}(\Omega)  $ the set of Borel probability measures on $\Omega$.  Given $\mu, \nu \in \mathcal{M}(\Omega)$, we define their mixed $K$-energy as
\begin{equation}\label{eq:2ContMutEnergy}
I_K(\mu, \nu) = \int\limits_{\Omega} \int\limits_{\Omega} K(x,y) \, d\mu (x) d\nu (y),
\end{equation}
and the  $K$-energy (also referred to as energy integral or energy functional) of $\mu$ to be
\begin{equation}\label{eq:2ContEnergy}
I_K(\mu) := I_K(\mu, \mu) =  \int\limits_{\Omega} \int\limits_{\Omega} K(x,y) \, d\mu (x) d\mu (y).
\end{equation}
We are interested in finding the optimal (maximal or minimal, depending on $K$) values of $I_K(\mu)$ over all $\mu \in \mathbb{P}(\Omega)$, as well as extremal measures for which these values are achieved, i.e.\ equilibrium measures with respect to $K$.

For a measure $\mu \in \mathbb{P}(\Omega)$,   the \textit{potential} $U_K^{\mu}$ of $\mu$ with respect to $K$, defined as
\begin{equation}\label{eq:2PotDef}
U_K^{\mu}(x) : = \int\limits_{\Omega} K(x, y )\, d \mu(y),\quad x\in \Omega,
\end{equation}
plays an important role in the study of energy minimization, see, e.g., Theorem \ref{thm:Constant on Supp}. 
%Notice that this meaning  of the term ``potential'' is consistent with us calling $K$ a potential, since the function $K( x, y )$ is simply the potential generated by a unit point charge at $y$, i.e. $ K( x, y ) = U_K^{\delta_y} (x)$.

It is well known that energy minimization is closely connected  to the positive definiteness of the  kernel. In Section \ref{sec:PosDefKer} we further explore this property  and  its  variants, such as conditional positive definiteness and  positive definiteness up to an additive constant. We discuss various relations between these properties and energy minimization, as well as inequalities for mixed energies, convexity of the energy functional, Hilbert--Schmidt operators, potential theory, etc. A synopsis of the main results presented in  Section \ref{sec:PosDefKer} may be found right before the beginning of \S\ref{sec:CPDPDC1}.  This section contains mostly background material (although several results do seem to be new) and is completely self-contained. 

%conditional positive definiteness, positive definiteness up to constant Among other things, we discuss we discuss the relation between conditional positive definiteness and  positive definiteness up to an additive constant  (Section  the equivalence of conditional positive definiteness to the arithmetic mean inequalities for mixed energies (Lemma \ref{lem:BHS2inputAInequality}  in Secion \ref{sec:MixedEnergy}), .   {sec:CPDPDC1}

In Section \ref{sec:InvMeasEnergy}, we  restrict our attention to the case when there exists a reasonable candidate  for  an energy minimizer - namely, a {\it{$K$-invariant measure}} (i.e., a measure whose potential $U_K^{\mu}$  with respect to $K$   is constant  on $\Omega$). This is a very natural class of measures which includes, for example, the uniform (Lebesgue) surface measure on the sphere $\mathbb S^{d-1}$ when the kernel $K$ is rotationally invariant. 
Assuming the existence of an invariant measure,  much more can be said about the topics of Section \ref{sec:PosDefKer}. In particular, conditional positive definiteness of $K$ is equivalent to the fact that $\mu$ minimizes $I_K$ over all {\it{signed measures of mass one}}  (Theorem \ref{thm:CPDequivMIN}),   conditional positive definiteness and  positive definiteness up to an additive constant are equivalent, which is not true in general (Lemma \ref{lem:CPDtoPDC}), and local minimizers are necessarily global (see \S\ref{sec:LocGlob}).  Many of the proofs in this section rely on a simple yet crucial identity \eqref{eq:Linear1} of Lemma \ref{lem:NuMinusMu}.

Finally, in Section \ref{sec:Energy Basics}, we further focus on the situation when there exists a $K$-invariant measure {\it{with full support}}.  These assumptions really tie the discussion together: Theorem \ref{thm:CPDEquivalences} states that in this case, conditional positive definiteness is equivalent to nine other natural properties, such as various versions of (local or global) convexity, mixed energy bounds, and  minimization of $I_K$ by $\mu$. Similar statements are provided for positive definiteness and  conditional strict  positive definiteness (Theorems \ref{thm:PDEquivalences} and \ref{thm:CSPDEquivalences}).

Admittedly, some portion of the results discussed above are well known in the field, see, for example, an  excellent discussion in \cite{BHS}. However, we undertook an extensive study of the literature, and it appears that many of the implications are in fact  new (we carefully point those out in the text), while some others are scattered in the literature. This paper presents a unified, comprehensive, and self-contained discussion of  both known and new connections, which results in an impressively long list of equivalent characterizations of conditional positive definiteness provided in Theorem \ref{thm:CPDEquivalences}.

Section \ref{sec:MetricStol} introduces an application of these results to discrepancy theory (see \S\ref{sec:EnStol} below), while Section \ref{sec:Sphere} specializes the prior discussion to the sphere, i.e. the case when $\Omega= \mathbb S^{d-1}$ and $K(x,y) = F (\langle x,y \rangle )$ is a rotationally invariant kernel. In this setting, our results recover and generalize various well known results.

\subsection{Energy and Discrepancy: the Stolarsky Invariance Principle.}\label{sec:EnStol}

Discrepancy is a classical way to assess the quality of a finite point distribution $\omega_N = \{z_1,\dots, z_N \} \subset \Omega$ by comparing its empirical measure to some chosen (usually uniform)  measure $\mu$   on some test sets.  Vast literature exists on discrepancy theory \cites{BC, Ma, KN}.

Discrepancy is closely connected to energy minimization. One of the first and most famous examples of this connection is given by the {\it{Stolarsky Invariance Principle}} \cite{St}, see \eqref{eq:StolOrig}, which connects the classical $L^2$ discrepancy with respect to spherical caps and the sum of Euclidean  distances, i.e.\ the discrete energy with the kernel $K(x,y) = \| x - y \|$, showing that maximizing the latter is equivalent to minimizing the former.  There has been increased interest in this principle in the recent years, including several new proofs and extensions to various settings and kernels, see, e.g., \cite{BrD,BDM,Sk1,Ba}. A more detailed discussion can be found in the beginning of Section \ref{sec:MetricStol}.

In Theorem \ref{thm:genStol}, we prove a very general  form of the Stolarsky Invariance Principle  on an arbitrary compact space $\Omega$ which connects the  (continuous or discrete) energy with a positive definite kernel $K$ to a notion of $L^2$ discrepancy based on the {\it{convolution square root}} of $K$, whose existence is equivalent to positive definiteness of $K$ (Proposition \ref{prop:Convolution}).   This generalizes several prior versions of this principle.

\subsection{Notation and conventions.} We shall always assume that $\Omega$ is
a compact metric space, although our discussion up to and including Theorem~\ref{thm:Constant on Supp} applies also to general topological measure spaces. As defined earlier, the class of
signed finite regular Borel measures on $\Omega$ will be denoted by $\mathcal M
(\Omega)$ and, in addition to the class $\mathbb P (\Omega)$ of Borel
probability measures, we shall consider the following subclasses of $\mathcal M
(\Omega)$:   the class $\widetilde{\mathbb P } (\Omega)$ of all {\it{signed}}
measures of total mass one and $\mathcal Z (\Omega)$ -- the class of  all
signed measures $\nu \in \mathcal M (\Omega)$ with mean zero, i.e. $\int_\Omega
d\nu =  \nu (\Omega) = 0$.  We shall say that two measures are equal if they
coincide on all Borel subsets of $\Omega$; { likewise, inequalities between
measures will be understood to hold on the Borel subsets of $\Omega  $.}

While a substantial portion of the theory applies to more general kernels, in the present text we restrict our attention  just to  continuous functions. Therefore, we shall say that $K$ is a {\it{kernel}} on $\Omega \times \Omega$ if $K:\, \Omega \times \Omega \rightarrow \mathbb R$ is  continuous and symmetric. This ensures that the energy $I_K (\mu)$ in \eqref{eq:2ContEnergy} is well defined for any measure $\mu \in \mathcal M (\Omega)$.

While we  will be interested in local and  global minimizers over different sets of measures,  whenever we say a measure $\mu$ is a {\it{minimizer}}  of $I_K$ without any additional information, we mean that $\mu$ is a global minimizer of $I_K$ over $\mathbb{P}(\Omega)$, i.e.\ for all $\nu \in \mathbb{P}(\Omega)$, $I_K(\mu) \leq I_K(\nu)$. {Unless explicitly stated otherwise,  local minimizers will be understood in the directional sense, see Definition \ref{def:locmin}. %topology induced by the total variation norm, denoted by $ \|\mu\| $.}  
The shorthand $\mu\restr_{B}$ will denote the restriction of $\mu$ to a set $B$, i.e. a measure defined by $\mu\restr_{B}  (A) = \mu (A\cap B) $. 

The Euclidean sphere in $\mathbb R^d$ will be denoted by $\mathbb S^{d-1}$ and $\sigma$ will denote the uniform (Lebesgue) measure on $\mathbb S^{d-1}$ normalized so that $\sigma (\mathbb S^{d-1})=1$.  

%%%%%%%%%%%%%%%%%%%%%

\section{Positive Definite Kernels and Energy Minimization}\label{sec:PosDefKer}

Positive definite kernels play an extremely important role in various areas of
mathematics, such as %\ov{I suggest that we remove the capitalization:} 
partial differential equations, machine learning, and
probability theory. In this discussion, we will focus on their relation to
energy minimization problems, but an exposition on their role in other areas
can be found, e.g., in \cites{A,F, Me}. 

%We shall always assume that $\Omega$ is a compact topological space. While a substantial portion of the theory applies to more general kernels, in the present text we restrict our attention  just to  continuous functions. Therefore, we shall say that $K$ is a {\it{kernel}} on $\Omega \times \Omega$ if $K:\, \Omega \times \Omega \rightarrow \mathbb R$ is  continuous and symmetric. 
We now  state the relevant definition in the form, which is most convenient for applications to energy optimization over Borel measures.
\begin{definition}\label{def:2PD}

A kernel $K: \Omega^2 \rightarrow \mathbb{R}$ is called \textbf{conditionally positive definite} if for every $\nu \in \mathcal{Z}(\Omega)$ (i.e.\ finite signed Borel measures satisfying $\nu(\Omega) = 0$), $I_K(\nu) \geq 0$.

If, moreover, $I_K(\nu) \geq 0$ for every finite signed Borel measure, i.e. $\nu \in \mathcal{M}(\Omega)$, then we call $K$ \textbf{positive definite}.

We call a kernel  \textbf{strictly positive definite} or \textbf{conditionally strictly positive definite} if it is positive definite or conditionally positive definite, respectively, and $I_K(\nu) = 0$ only if  $\nu = 0$. %$\nu(A) = 0$ for all Borel sets $A \subseteq \Omega$.

If there exists some $C \in \mathbb{R}$ such that $K+C$ is a (strictly) positive definite kernel, we call $K$ \textbf{(strictly) positive definite modulo an additive constant} (or \textbf{up to an additive constant}).
\end{definition}

A more standard way of defining positive definiteness of a kernel $K: \Omega^2 \rightarrow \mathbb{R}$ is by requiring that, for every $N \in \mathbb{N}$ and $\{z_i \}_{i=1}^N \subset \Omega$, the matrix $ \big[ K \big(  z_i, z_j  \big) \big]_{i,j=1}^N$ is positive semidefinite, i.e.\ for any sequence $\{c_i\}_{i=1}^N \subset \mathbb{R}$, the kernel $K$ satisfies the inequality
\begin{equation}\label{eq:PDDefMatrix}
\sum\limits_{i,j=1}^N c_i c_j K \big( z_i, z_j  \big)\geq 0.
\end{equation}
Since the kernel $K$ is continuous, this is clearly equivalent to Definition \ref{def:2PD} due to the weak$^*$ density of discrete measures in $\mathcal{M}(\Omega)$. Similarly, conditional positive definiteness is equivalent to \eqref{eq:PDDefMatrix}     with the additional condition $\sum c_i =0$.  We finally  remark that such an equivalence does not hold for the  strict version of these properties. 

%Occasionally in this text, especially when focusing on energy maximization, we will also want to make use of the negative definiteness of a kernel. We shall call a kernel $K$ \textit{negative definite} if the kernel $-K$ is positive definite

A constant positive kernel, i.e. $K(x,y) = c > 0$ for all $x, y \in \Omega$, is necessarily positive definite, hence   such kernels always exist.   Moreover, the class of positive definite kernels is easily seen to be closed under addition, multiplication (a result known as the Schur product theorem), and  limits of uniformly convergent sequences.

\begin{lemma}\label{lem:PDArithmetic}
If $K$ and $L$ are positive definite kernels on $\Omega$, then so are $K+L$ and $KL$. If $K_1, K_2, ...,$ are positive definite and $\lim_{n \rightarrow \infty} K_n = K$ uniformly, then $K$ is positive definite. The statements regarding the sum and limit (but not the product) hold if we replace positive definiteness with conditional positive definiteness.
\end{lemma} 

For a continuous function $\phi :\Omega \rightarrow \mathbb R$, the kernel $\phi (x) \phi (y)$ is obviously positive definite, since 
$$  \int\limits_{\Omega} \int\limits_{\Omega} \phi (x) \phi(y)\, d \mu(x) d\mu (y)  = \Bigg( \int\limits_{\Omega} \phi (x)  d \mu(x)  \Bigg)^2 \ge 0. $$
Therefore, Lemma \ref{lem:PDArithmetic} provides a rich class of examples of positive definite kernels.

\begin{lemma}\label{lem:PDconstructed}
For $j \in \mathbb{N}_{0}$, let $\lambda_j \geq 0$ and $\phi_j: \Omega \rightarrow \mathbb{R}$ be continuous. Then if the series converges absolutely and uniformly, the kernel
\begin{equation}\label{eq:PDconstructed}
K(x,y) = \sum_{j=0}^{\infty} \lambda_j \phi_j(x) \phi_j(y)
\end{equation}
%(which is continuous, due to uniform convergence) 
is positive definite.
\end{lemma}

%\begin{proof}
%Let $\mu \in \mathcal{M}(\Omega)$. Then
%\begin{align*}
%I_K(\mu) & = \sum_{j=0}^{\infty} \lambda_j \int\limits_{\Omega} \int\limits_{\Omega} \phi_j(x) \phi_j(y) d \mu(x) d\mu(y) \\
%& = \sum_{j=0}^{\infty} \lambda_j \Big( \int\limits_{\Omega}  \phi_j(x) d \mu(x) \Big)^2 \\
%& \geq 0.
%\end{align*}
%\end{proof}

In fact,  the well-known Mercer's Theorem (see Theorem \ref{thm:Mercer} in  Section \ref{sec:HS}) demonstrates that the representation \eqref{eq:PDconstructed} actually provides a characterization of positive definite kernels, see Corollary \ref{cor:PDCharacterization}. \\

%In addition to a relatively simple characterization, positive definite kernels have several other useful properties, many of which will be discussed in Chapter \ref{chp:Energy Basics}. To begin with, they are closed under addition, multiplication (a result known as the Schur product theorem), and limits of uniformly convergent sequences.

In what follows, we  provide various  results which connect properties of the kernel $K$, the energy  functional $I_K$, and the minimizers of this  energy integral.   While some of these results  have previously appeared in the literature, a number of them  seem to be  new (we shall specifically point them out in the exposition below). In addition, it seems that even the known results have not all appeared simultaneously in a single text  (perhaps the most complete   prior exposition of this kind  is the  discussion of  lower semi-continuous kernels on compact sets in  \cite[Chapter 4]{BHS}). Moreover, the general results of this section form a basis for the long list of equivalences provided in Theorems \ref{thm:CPDEquivalences}, \ref{thm:PDEquivalences}, and \ref{thm:CSPDEquivalences} under some additional assumptions.  Therefore, for the sake of making our presentation self-contained and coherent, we have decided to include all the relevant background information (not just the new results) in this section. \\

%%%%%%%%%%%%%%%%%%%%%

%\ov{Maybe we could pick the more important parts from this overview and move them to Introduction? It may appear excessive to give overview of individual sections --- this is really not that long a paper ;-)}
In Section \ref{sec:CPDPDC1}, we explore the relation between conditional positive definiteness and positive definiteness up to an additive constant. We show that the latter implies the former (Lemma \ref{lem:PD-CPD}), but not vice versa (Example \ref{ex:x1+y1}). This relation will be revisited in Section \ref{sec:CPDDPC2}, where the converse implication is established under additional assumptions (Lemma \ref{lem:CPDtoPDC}). 

Section \ref{sec:MixedEnergy} discusses the equivalence of conditional positive definiteness and the arithmetic-mean inequality for mixed energies (Lemma  \ref{lem:BHS2inputAInequality}), as well as the similar equivalence between positive definiteness and the  geometric-mean inequality, Lemma \ref{lem:BHS2inputGInequality}. While the fact that (conditional) positive definiteness implies such  inequalities is well-known, we did not find the converse implication  in the literature.

In Section \ref{sec:PDConv}, we concentrate on the interplay between conditional positive definiteness of the kernel $K$ and the convexity of the corresponding energy functional $I_K$. In particular, Proposition \ref{prop:ConvexCPDEqual2} demonstrates that the two notions are equivalent. Again, only one direction seems to have appeared in the literature before.

Section \ref{sec:MinMeas} reviews some basic facts about the potential of the global and local   minimizers of the energy integral $I_K$ (Theorem \ref{thm:Constant on Supp} and Corollary \ref{cor:Constant on Supp}, respectively): if $\mu$ is a (local) minimizer of $I_K$, then  $U_K^{\mu}$ is constant on $\operatorname{supp}(\mu)$.

In Section \ref{sec:HS}, we recall the connection between the positive definiteness of the kernel $K$ and  the properties of the generated Hilbert--Schmidt operator $T_{K,\mu}$. Lemma  \ref{lem:Hilbert-Schmidt} demonstrates the equivalence between positive definiteness of $K$ and positivity of $T_{K,\mu}$, while Mercer's Theorem (Theorem \ref{thm:Mercer}) provides the absolutely and uniformly  convergent expansion of a positive definite kernel in term of the eigenfunctions of the associated Hilbert--Schmidt operator. 

In Section  \ref{sec:SqRoot}, we demonstrate the existence of the ``convolution square root'' of a positive definite kernel (see \eqref{eq:Convolution} in Proposition \ref{prop:Convolution}). In the  case  $\Omega = \mathbb S^{d-1}$, this fact has been observed in \cites{BD, BDM}, but the general case presented here is new. 

Finally, Section \ref{sec:MinHS} explores the relation between energy minimizers and Hilbert--Schmidt operators. In particular, Lemma \ref{lem:OperatorPos} shows that  if $\mu$ minimizes  $I_K$, then the associated operator $T_{K,\mu}$ is positive, which leads to an  important fact (Lemma \ref{lem:PosDefonSupp}): if $\mu$ is a (local) minimizer of $I_K$, then $K$ is positive definite (up to a constant) on the support of $\mu$. Similar results have appeared in  various papers on energy minimization, e.g.  \cites{CFP, FS}.

\subsection{Conditional Positive Definiteness and Positive Definiteness up to a Constant}\label{sec:CPDPDC1}

Since adding a constant to a kernel obviously does not affect the minimizers,  it is natural to  consider kernels that are (strictly) positive definite modulo a constant. However, adding a constant also  never changes conditional (strict) positive definiteness, as for all $C \in \mathbb{R}$ and $\nu \in \mathcal{Z}(\Omega)$,
\begin{equation*}
I_{K+C}(\nu) = I_K(\nu) + (\nu(\Omega))^2 C = I_K(\nu).
\end{equation*}
Since (strict) positive definiteness implies conditional (strict) positive definiteness, we arrive at the following lemma.

\begin{lemma}\label{lem:PD-CPD}
If $K$ is (strictly) positive definite modulo a constant, then $K$ is conditionally (strictly) positive definite.
\end{lemma}

In Section \ref{sec:InvMeasEnergy} (Lemma \ref{lem:CPDtoPDC}), we will demonstrate that the converse of Lemma \ref{lem:PD-CPD} can hold under certain conditions. However, it does not hold in general: 
\begin{example}\label{ex:x1+y1}
Consider $K: \mathbb{S}^{d-1}\times  \mathbb{S}^{d-1} \rightarrow \mathbb{R}$ defined by $K(x,y) = x_1 + y_1$, where $x_1 = \langle x, e_1 \rangle$. Then $K$ is conditionally positive definite, but not positive definite modulo a constant.
\end{example}

\begin{proof}
For all $\nu \in \mathcal{Z}(\mathbb{S}^{d-1})$,
\begin{equation*}
I_K(\nu) = 2 \int\limits_{\mathbb{S}^{d-1}} \int\limits_{\mathbb{S}^{d-1}}  x_1 \,d\nu(x) d\nu(y) = 0,
\end{equation*}
so $K$ is conditionally positive definite.

Now we show there is no constant $C$ such that $K+C$ is positive definite. If $C < 0$, then
\begin{equation*}
I_{K+C}(\sigma) = 2 \int\limits_{\mathbb{S}^{d-1}} x_1 d \sigma(x) + C = C < 0.
\end{equation*}
Suppose that $C \geq 0$ and let $\mu = (C+1)\delta_{-e_1} - C \delta_{e_1} \in \mathcal{M}(\mathbb{S}^{d-1})$. Then
\begin{align*}
I_{K+C}(\mu)& = 2 \mu(\mathbb{S}^{d-1}) \int\limits_{\mathbb{S}^{d-1}} x_1\, d\mu(x) + C (\mu(\mathbb{S}^{d-1}))^2 \\
& = 2 (-2C-1) + C  = -3C-2 < 0.
\end{align*}
The proof is now complete.
\end{proof}

%\section{Mercer's Theorem}\label{sec:MerThm}

%%%%%%%%%%%%%%%%%%%%%%%%%%%%%%%%%%%%

%%%%%BREAKBREAKBREAK

\subsection{Positive Definiteness and Inequalities for Mixed Energies}\label{sec:MixedEnergy}

We first make the observation that the (conditional) positive definiteness of the kernel can be characterized by the inequalities for mixed energies in terms of arithmetic or geometric means. While the validity of such inequalities for positive definite kernels is well known, see, e.g.,  \cite[Chapter 4]{BHS}, their sufficiency doesn't seem to have appeared in previous literature.   We summarize these facts in the following two lemmas. The first one connects conditional positive definiteness to the arithmetic mean inequality. 

\begin{lemma} \label{lem:BHS2inputAInequality} 
Suppose $K$ is a kernel on $\Omega \times \Omega$. Then the following conditions are equivalent:
\begin{enumerate}
\item \label{cpd1} $K$ is conditionally positive definite.
\item \label{cpd2} For every pair of Borel probability measures $\mu_1$ and $\mu_2$  on $\Omega$, the mutual energy $I_K(\mu_1,\mu_2)$ satisfies
\begin{equation}\label{eq:2inputAInequality}
I_K(\mu_1,\mu_2)\leq \frac12 \big( I_K(\mu_1)+I_K(\mu_2) \big).
\end{equation}
\item  \label{cpd3} Inequality \eqref{eq:2inputAInequality} is satisfied for any pair of signed Borel measures of total  mass one. 
\end{enumerate}

%If $K$ is conditionally strictly positive definite, then equality in \eqref{eq:2inputAInequality} holds if and only if $\mu_1=\mu_2$. % on Borel subsets of $\Omega$.

In addition, conditional strict positive definiteness of $K$ is equivalent to the fact that the inequality in  \eqref{eq:2inputAInequality} is strict unless $\mu_1=\mu_2$.
\end{lemma}

%\textcolor{red}{One can also show that if equality in \eqref{eq:2inputAInequality} holds if and only if $\mu_1 = \mu_2$, then $K$ is conditionally strictly positive definite, as we would have
%$$ I_K(c(\mu_1 - \mu_2)) = c^2 \Big( I_K(\mu_1) - 2 I_K(\mu_1, \mu_2) + I_K(\mu_2) \Big) \geq 0$$
%with equality if and only if $\mu_1 = \mu_2$ or $c = 0$.
%}

\begin{proof}
Suppose that $K$ is conditionally positive definite. Then for any $\mu_1, \mu_2 \in \widetilde{\mathbb{P}}(\Omega)$, $\mu_1 - \mu_2 \in \mathcal{Z}(\Omega)$, so

\begin{equation*}
0\leq  I_K( \mu_1 - \mu_2) = I_K(\mu_1) - 2 I_K(\mu_1, \mu_2) + I_K(\mu_2),
\end{equation*}
which proves \eqref{eq:2inputAInequality}. Thus, \eqref{cpd1} implies \eqref{cpd3}, which in its turn obviously implies \eqref{cpd2}.

Now assume condition \eqref{cpd2}, i.e. that \eqref{eq:2inputAInequality} holds for all $\mu_1, \mu_2 \in \mathbb{P}(\Omega)$. For any $\mu \in \mathcal{Z}(\Omega)$, there exists $c \geq 0$ and probability measures $\mu_1, \mu_2 \in \mathbb{P}(\Omega)$ such that $\mu = c( \mu_1 - \mu_2)$. We then have that
\begin{equation}\label{eq:qqq}
I_K(\mu) = I_K \big( c( \mu_1 - \mu_2)\big) = c^2 \Big( I_K(\mu_1) - 2 I_K(\mu_1, \mu_2) + I_K(\mu_2) \Big) \geq 0,
\end{equation}
so $K$ must be conditionally positive definite.

%Let us now assume that $K$ is conditionally strictly positive definite. Equality clearly holds in \eqref{eq:2inputAInequality} if $\mu_1 = \mu_2$. % on all Borel sets of $\Omega$. 
%If, however, %the restriction of $\mu_1$ to Borel subsets of $\Omega$ does not coincide with that of $\mu_2$, 
%the measure $\mu_1 - \mu_2 $ is nonzero,  conditional strict positive definiteness of $K$ implies that  $I_K(\mu_1 - \mu_2) > 0$, making  inequality \eqref{eq:2inputAInequality}  strict.
If $K$ is conditionally strictly positive definite and the measure $\mu_1 - \mu_2 $ is nonzero, then $I_K(\mu_1 - \mu_2) > 0$, making  inequality \eqref{eq:2inputAInequality}  strict. Conversely, if  \eqref{eq:2inputAInequality} is strict whenever $\mu_1 - \mu_2 \neq 0$, then equality in  \eqref{eq:qqq} can hold only if $\mu_1 = \mu_2$ or $c = 0$, i.e. $K$ is conditionally strictly positive definite.
\end{proof}

The second lemma is very similar: it shows that positive definiteness is equivalent to  the geometric mean inequality for the mixed energy. 

\begin{lemma}\label{lem:BHS2inputGInequality}

Suppose $K$ is a kernel on $\Omega \times \Omega$. Then $K$ is positive definite if and only if $I_K (\mathbb P (\Omega)) \subset [0,\infty)$ and  for all $\mu_1, \mu_2 \in \mathbb{P}(\Omega)$, the mutual energy $I_K(\mu_1,\mu_2)$ satisfies
\begin{equation}\label{eq:2inputGInequality}
I_K(\mu_1,\mu_2)\leq \sqrt{I_K(\mu_1) I_K(\mu_2)},
\end{equation}
%and $I_K(\mu_1) \geq 0$.
 $K$ is strictly positive definite if and only if  the  inequality in \eqref{eq:2inputGInequality}  is strict unless $\mu_1=\mu_2$. % on Borel subsets of $\Omega$. 
\end{lemma}

Observe that the proof below shows that, just like in Lemma \ref{lem:BHS2inputAInequality}, we could replace ${\mathbb{P}}(\Omega)$ with $\widetilde{\mathbb{P}}(\Omega)$.  Heuristically, this statement  says that the bilinear form $I_K (\mu_1,\mu_2)$ defines an inner product on measures if and only if $K$ is positive definite, and \eqref{eq:2inputGInequality} is just the Cauchy--Schwarz inequality.

\begin{proof}
Suppose that $K$ is  positive definite. For any $t \in \mathbb{R}$ and $\mu_1, \mu_2 \in \mathbb{P}(\Omega)$,  we define %$t\mu_1 - \mu_2 \in \mathcal{M}(\Omega)$ so
\begin{equation*}
g(t) := t^2 I_K(\mu_1) - 2t I_K(\mu_1, \mu_2) + I_K(\mu_2) = I_K( t \mu_1 - \mu_2) \geq 0.
\end{equation*}
Thus the discriminant $4 I_K( \mu_1, \mu_2)^2 - 4 I_K(\mu_1) I_K(\mu_2)$ of the quadratic polynomial $g(t)$ is nonpositive, which yields  \eqref{eq:2inputGInequality}. If $K$ is strictly positive definite, $g(t) = 0$ has a root only if  $t \mu_1 - \mu_2 = 0$, which implies $t=1$ and  $\mu_1 = \mu_2$.

Suppose instead that \eqref{eq:2inputGInequality} holds for all probability measures. For any $\mu \in \mathcal{M}(\Omega)$, there exists $a, b \geq 0$ and $\mu_1, \mu_2 \in \mathbb{P}(\Omega)$ such that $\mu = a\mu_1 - b\mu_2$. We then have that
\begin{equation*}
I_K(\mu) = a^2 I_K(\mu_1) - 2 a b I_K(\mu_1, \mu_2) +  b^2 I_K(\mu_2)  \geq ( a \sqrt{I_K(\mu_1)} - b \sqrt{I_K(\mu_2)})^2 \geq 0,
\end{equation*}
implying that  $K$ is positive definite. If \eqref{eq:2inputGInequality} is strict unless $\mu_1 = \mu_2$, then the inequality above is strict unless $\mu=0$, i.e. $K$  is strictly  positive definite.% Conversely, if  $K$ is strictly positive definite,  equality  in \eqref{eq:2inputGInequality} implies that $g(t)$ has a unique root $t'$. Since $K$ is strictly positive definite, we have $\mu_1 = t' \mu_2$ on Borel subsets of $\Omega$ and since $\mu_1, \mu_2 \in \mathbb{P}(\Omega)$, we see that $t' = 1$ and $\mu_1 = \mu_2$.
\end{proof}

% Finally, we use the arithmetic mean inequality of Lemma \ref{lem:BHS2inputAInequality} to show that minimizers are unique for strictly conditionally positive definite kernels. 

% \begin{theorem}\label{thm:CSPDUniqueMin}
% Suppose that $K$ is conditionally strictly positive definite. Then $I_K$ has a unique minimizer (either in ${\mathbb{P}}(\Omega)$ or in $\widetilde{\mathbb{P}}(\Omega)$).
% \end{theorem}

% \begin{proof}
% Suppose that $\mu_1, \mu_2 \in \mathbb{P}(\Omega)$ (the argument for $\widetilde{\mathbb{P}}(\Omega)$ is identical) are equilibrium measures of $I_K$, i.e. energy minimizers. Due to Lemma \ref{lem:BHS2inputAInequality} the probability measure $\mu = \frac{1}{2}( \mu_1 + \mu_2)$ satisfies
% \begin{equation}
% I_K(\mu) = \frac{1}{4} I_K(\mu_1) + \frac{1}{2} I_K(\mu_1, \mu_2) + \frac{1}{4} I_K(\mu_2) \leq \frac{1}{2} I_K(\mu_1) + \frac{1}{2} I_K(\mu_2) =  \inf_{\nu \in \mathbb P (\Omega) } I_K (\nu). % \mathcal{I}_K(\Omega).
% \end{equation}
% Thus, $\mu$ must also be a minimizer of $I_K$, which means that $  2 I_K(\mu_1, \mu_2) = I_K(\mu_1) + I_K(\mu_2)$. Lemma \ref{lem:BHS2inputAInequality} then implies that $\mu_1 = \mu_2$, %on Borel subsets of $\Omega$, which proves the  claim. 
% \end{proof}

\subsection{Positive Definiteness and Convexity of the Energy Functional}\label{sec:PDConv}

Convexity plays an important role in optimization problems, and so do various versions of positive definiteness.  Therefore, it is not surprising that the two notions are related. In fact, as we shall see, in many settings, they are equivalent. %In particular, conditional positive definiteness of $K$ is equivalent to the convexity 

Such  equivalences, in different forms, have previously  appeared in the literature  \cite{BFGMPV,CSh,DPZ,Mec,P,PZ,ZDP}. This connections appears to be common knowledge in some references, but is largely  overlooked in many other. In this subsection, we take a deeper look at this phenomenon.

\begin{definition}
    \label{def:convexity}
Let  $K: \Omega \times \Omega  \rightarrow \mathbb{R}$ be a kernel. We say that $I_K$ is \textbf{convex at} $\mu \in \mathbb{P} (\Omega) $ if for every  $\nu \in \mathbb{P}(\Omega)$ there exists some $t_\nu \in (0,1]$  such that for all $t \in [0,t_\nu)$
\begin{equation}\label{eq:2inputConvexDef}
I_K((1-t) \mu + t \nu) \leq (1-t) I_K(\mu) + t I_K(\nu).
\end{equation}
We say $I_K$ is \textbf{convex} on $\mathbb{P}(\Omega)$ if inequality \eqref{eq:2inputConvexDef} holds for every $\mu$, $\nu \in \mathbb{P}( \Omega)$ and  all $t\in[0,1]$; it is said to be \textbf{strictly convex}, if the inequality is strict for all $ t\in (0,1) $ unless $ \mu = \nu $.
\end{definition}

Similarly to the above definition, we consider (strict) convexity on $ \widetilde {\mathbb P}(\Omega) $ and $ \mathcal M(\Omega) $. Unless noted otherwise, convexity is understood on $ \mathbb P(\Omega) $.

We observe that convexity of $I_K$ on $\mathbb{P}(\Omega)$ is equivalent to the fact that $I_K$  is  convex at all $\mu \in \mathbb{P} (\Omega)$. Indeed, if \eqref{eq:2inputConvexDef} fails for some  $\mu$, $\nu \in \mathbb{P}( \Omega)$ and some $t\in (0,1)$, then the quadratic polynomial $f(t) =  I_K((1-t) \mu + t \nu)$ is not convex on the  interval $[0,1]$, i.e.   $f''(t) <0$ for all   $t\in [0,1]$ and  $I_K$ fails to be convex at  $\mu$. \\

We first show that convexity is equivalent to the arithmetic mean inequality \eqref{eq:2inputAInequality} for mixed energies.

\begin{lemma}\label{lem:ConvexEqual2}
The energy functional $I_K$ is convex at $\mu \in \mathbb{P}(\Omega)$ if and only if for all $\nu \in \mathbb{P}(\Omega)$,
\begin{equation}\label{eq:AM2}
I_K(\mu, \nu)\leq \frac12 \big( I_K(\mu)+I_K(\nu) \big).
\end{equation}
Consequently, $I_K$ is convex on  $\mathbb{P}(\Omega)$ if and only if  inequality \eqref{eq:AM2} holds for all  $\mu, \nu \in \mathbb{P}(\Omega)$.

In addition, strict convexity of $ I_K $ is equivalent to the above inequality being strict unless $\mu=\nu$.
\end{lemma}

\begin{proof}
Let $\nu \in \mathbb{P}(\Omega)$ and assume  that the arithmetic-mean inequality \eqref{eq:AM2} holds. Then for all $t \in [0,1]$, \begin{equation}\label{eq:ConvexEqual2}
I_K( (1-t) \mu + t \nu) = (1-t)^2 I_K(\mu) + 2 (1-t) t I_K(\mu, \nu) + t^2 I_K(\nu) \leq (1-t) I_K(\mu) + t I_K(\nu).
\end{equation}
So $I_K$ is indeed convex at $\mu$.

For the converse direction, assume that $I_K$ is convex at $\mu$. Then for any $\nu \in \mathbb{P}(\Omega)$ and  $t>0$ sufficiently small, inequality \eqref{eq:ConvexEqual2} holds, so
\begin{equation*}
2 (1-t) t I_K(\mu, \nu) \leq t(1-t) (I_K(\mu) + I_K(\nu)).
\end{equation*}
Dividing by $t (1-t)$, we obtain the  arithmetic mean inequality \eqref{eq:AM2}. Lastly, for  strictly convex $ I_K $, the proof is the same, with all inequalities strict unless $ \mu = \nu $.
\end{proof}

Observe that one can easily replace $\mathbb P (\Omega)$ with $\widetilde{\mathbb P}(\Omega)$ in Lemma \ref{lem:ConvexEqual2}. At the same time, according to parts \eqref{cpd2}-\eqref{cpd3} of Lemma \ref{lem:BHS2inputAInequality}, the validity of \eqref{eq:AM2} on  $\mathbb P (\Omega)$ is equivalent to its validity on  $\widetilde{\mathbb P}(\Omega)$. Hence we obtain the following corollary. 

\begin{corollary}\label{c:convexequiv}
$I_K$ is (strictly) convex on ${\mathbb P}(\Omega)$ if and only if it is (strictly) convex on $\widetilde{\mathbb P}(\Omega)$. 
\end{corollary}

Lemmas \ref{lem:BHS2inputAInequality} and \ref{lem:ConvexEqual2} together clearly imply the desired equivalence between convexity and conditional positive definiteness:

\begin{proposition}\label{prop:ConvexCPDEqual2}
Let $K: \Omega \times \Omega  \rightarrow \mathbb{R}$ be  a kernel. The kernel  $K$ is conditionally (strictly) positive definite if and only if the energy functional $I_K$ is (strictly) convex on $\mathbb{P}(\Omega)$ (or, equivalently, on $\widetilde{\mathbb P}(\Omega)$). 
\end{proposition}

%We also remark that Lemma \ref{lem:ConvexEqual2} may be easily adapted to show that strict convexity of $I_K$ is equivalent to the strict arithmetic mean inequality in \eqref{eq:AM2}  for $\mu \neq \nu$.  
%In turn, due to Lemma \ref{lem:BHS2inputAInequality}, this leads to the 

This equivalence between  convexity of $I_K$ on $\mathbb{P}(\Omega)$ and conditional  positive definiteness of $K$ has  been observed, e.g.  in  \cite{BFGMPV,CSh,P,PZ}.  In fact,  it admits a short direct proof (see Proposition \ref{p:mircea} and the discussion thereafter). We chose to prove Lemma \ref{lem:ConvexEqual2} first (an approach taken in a recent paper of the authors, joint with  A.~Glazyrin, D.~Ferizovi\'c, and J.~Park \cite{BFGMPV}), since it also allows us to establish equivalences between local versions of  properties, which will be important later, see e.g. parts  \eqref{P7} and \eqref{P9} of Theorem \ref{thm:CPDEquivalences}. 
% (and in\cite{PZ} for strict versions).

\begin{corollary}\label{cor:CSPDUniqueMin}
Suppose that $K$ is conditionally strictly positive definite. Then $I_K$ has a unique minimizer (either in ${\mathbb{P}}(\Omega)$ or in $\widetilde{\mathbb{P}}(\Omega)$).
\end{corollary}

\begin{proof}
   % By parts \eqref{cpd2}-\eqref{cpd3}  of 
   According to Lemmas~\ref{lem:BHS2inputAInequality} and \ref{lem:ConvexEqual2}, strict conditional positive definiteness of $K$  implies  strict convexity of $ I_K $ on both ${\mathbb{P}}(\Omega)$ and $\widetilde{\mathbb{P}}(\Omega)$. %; the argument for strict properties is identical. 
   To conclude, observe that these measure spaces are themselves convex.
\end{proof}

Finally, we note that  convexity of $ I_K $  %in the sense of Definition~\ref{def:convexity} can be considered 
on the entire  space $ \mathcal M(\Omega) $ is equivalent to positive definiteness of the kernel.  More precisely, we have the following.
\begin{proposition}[{\cite[Lemma 4]{P}}]\label{p:mircea}
    The functional $ I_K $ is (strictly) convex on $ \mathcal M(\Omega) $ if and only if $ K $ is  (strictly) positive definite on $ \Omega $.
\end{proposition}
%The proof of this statement is similar to that of Lemmas~\ref{lem:BHS2inputAInequality} and \ref{lem:ConvexEqual2}.
We include a simple proof of this proposition for the sake of completeness. 
\begin{proof}
Simply observe that for any $t\in (0,1)$, the inequality  
$$ I _K ((1-t) \mu + t \nu ) = (1-t)^2 I_K (\mu) + 2 t(1-t) I_K (\mu,\nu) + t^2 I_K (\nu) \le (1-t) I_K (\mu)  + t I_K (\nu)$$
is equivalent to 
$$  t(1-t) \big( I_K (\mu ) - 2 I_K(\mu,\nu) + I_K (\nu) \big) = t(1-t) I_K (\mu - \nu) \ge 0, $$ from which the statement easily follows. 
\end{proof}
We note that this argument also gives an alternative proof of Proposition \ref{prop:ConvexCPDEqual2}, since any measure with mass zero can be represented as a multiple of the difference of two measures with mass one.

\subsection{Minimizing Measures: Basic Potential Theory}\label{sec:MinMeas}
It is well known that the behavior of the minimizing measures is closely connected to the behavior of the  potential of the minimizing measure  with respect to the kernel. For a detailed account of the topic, we refer the reader to Chapter 4 of  \cite{BHS}. The following simple  statement is classical, see, e.g., \cite{Bj}. We provide its proof for completeness.

\begin{theorem}\label{thm:Constant on Supp}
Suppose that $\mu$ is a minimizer of $I_K$ over $\mathbb{P}(\Omega)$. Then $U_K^{\mu}(x) = I_K(\mu)$ on $\operatorname{supp}(\mu)$ and $U_K^{\mu}(x) \geq I_K(\mu)$ on $\Omega$.
\end{theorem}

\begin{proof} 
  %  \ov{rewrote the proof a bit and fixed 1 typo; hopefully no new ones}
Let $\nu \in \mathcal{Z}(\Omega)$ be such that $\mu + \varepsilon \nu \geq 0$ for all  $0 < \varepsilon \leq \epsilon_0 $ with some positive $ \epsilon_0 $. This clearly means that $\mu + \varepsilon \nu \in \mathbb{P}(\Omega)$, so
\begin{equation}
    \label{eq:muisminimum}
I_K(\mu) \leq I_K(\mu + \varepsilon \nu)  = I_K(\mu) + 2 \varepsilon I_K(\mu, \nu) + \varepsilon^2 I_K(\nu).
\end{equation}
Thus, for $0 \leq \varepsilon \leq \epsilon_0 $,
$$ 0 \leq \varepsilon \left( 2  I_K(\mu, \nu) + \varepsilon I_K(\nu) \right),$$
implying $I_K( \mu, \nu) \geq 0$.

Suppose that there exist $c_1,c_2 \in \mathbb{R}$, $z \in \operatorname{supp}(\mu)$ and $y \in \Omega$ such that
$$ c_1 = U_K^{\mu}(y) < U_K^{\mu}(z) = c_2.$$
Let $B_z$ be a ball centered at $z$, so small that $ \min_{x \in B_z} U_K^{\mu}(x) > (c_1+c_2)/2 $.
Define 
\begin{equation}\label{eq:2nu1}
 \nu := \mu(B_z)\cdot \delta_{y} - \mu\restr_{B_z}.
 \end{equation}
 Then $ \nu \in \mathcal Z(\Omega) $ and satisfies $\mu + \varepsilon \nu \geq 0$ for any $ 0\leq \epsilon \leq 1 $. On the other hand, 
\begin{equation*}
    I_K( \mu, \nu) = \mu(B_z) \cdot U_K^{\mu}(y) - \int\limits_{B_z} U_K^{\mu}(x) \,d\mu(x)  \leq  \mu(B_z)\cdot c_1 - \mu(B_z)\cdot\frac{c_1+c_2}2 < 0,
\end{equation*}
which is a contradiction. Thus, necessarily $ U_K^{\mu}(y) \geq U_K^{\mu}(z) $ for $ y,z $ as above. Since $ y $ can belong to $ \operatorname{supp}(\mu) $ and $  \int_\Omega U_K^{\mu}(x)\, d\mu(x) =I_K(\mu)  $, both claims of the theorem follow.
\end{proof}

Notice that if $\mu$ has full support, i.e.\ $\operatorname{supp} (\mu) = \Omega$, the conclusion of Theorem \ref{thm:Constant on Supp} states that the potential $U_K^\mu (x)$ is constant on $\Omega$. Measures with constant potentials  will be further explored in Section \ref{sec:InvMeasEnergy}.

\begin{definition}\label{def:locmin}
  %  \ov{rewrote the definition}
    We shall say that $\mu$ is a \textbf{local minimizer} of $I_K$ in $\mathbb P (\Omega)$ \textbf{with respect to a given metric}  $d (\cdot , \cdot) $ on $\mathcal M (\Omega)$   if it is a local minimizer in  the topology induced by this metric, in other words, if there exists $\epsilon_\mu > 0$, 
    such that  for all $ \nu \in \mathbb P(\Omega) $ satisfying $ d(\mu,\nu) \leq \epsilon_\mu $, we have  %and $ \mu+\nu \in\mathbb P(\Omega) $,  we have 
    $$ I_K (  \mu + \nu ) \geq I_K (\mu).$$
    
   We shall say that $\mu$ is a \textbf{directional local minimizer} of $I_K$ in $\mathbb P (\Omega)$ if it is a local minimizer in every direction, i.e., if  for each $\nu \in \mathbb{P} (\Omega)$, there exists $\tau_\nu  \in (0,1]$  such that  for all $t \in [0, \tau_\nu]$ we have 
$$ I_K \big(  (1-t) \mu + t \nu \big) \geq I_K (\mu).$$

\end{definition}

The difference between the two definitions above  is similar to that between the Gateaux and Fr\'{e}chet derivatives. 

%\ov{rewrote this paragraph}
Observe that, if $\Omega$ is a compact metric space, % due to the compactness of $ \Omega $, 
convergence in total variation implies convergence in  Wasserstein $W_p$ metric for $ 1\leq p < \infty  $. Thus, a  local minimizer in one of $ W_p $ is also a local minimizer with respect to the total variation distance. % In the present work, we shall always use the words {\it{local minimizers}} in the sense of Definition \ref{def:locmin}. %The following proposition provides a relation between the local and global minimizers.

In turn,  for $\mu,\nu \in \mathbb P (\Omega)$, we have $ \big( (1-t) \mu + t \nu \big) - \mu =  t (\nu - \mu) $, and the total variation norm  satisfies $\| t (\nu - \mu) \|_{TV} \le 2t < \varepsilon_\mu$ for $t$ small enough. Therefore, local minimizers in total variation are directional local minimizers (but not vice versa).  This is summarized below:
\begin{equation}\label{eq:LocMinRel}
\begin{pmatrix}
\textup{local minimizer in }\\ {W_p, \,\, 1\le p<\infty}\\
\end{pmatrix}  \Longrightarrow 
\begin{pmatrix}
\textup{local minimizer in }\\ { \textup{  total variation}} \\
\end{pmatrix} 
 \Longrightarrow 
 \begin{pmatrix}
 \textup{directional}\\ \textup{local minimizer }\\ 
\end{pmatrix}.
\end{equation}
In this text, unless explicitly specified otherwise, the words ``local minimizer'' in the assumption of a statement will mean the directional minimizer, as this is the weakest assumption, and thus  corresponding results will also hold for the other types of local minimizers mentioned in \eqref{eq:LocMinRel}.

Analyzing the proof of Theorem \ref{thm:Constant on Supp}, we find that for $\nu$ defined in \eqref{eq:2nu1}, we can write $\mu +  \varepsilon \nu  = (1- \varepsilon) \mu + \varepsilon \widetilde{\nu}$ with $\widetilde{\nu} = \mu + \nu \in \mathbb{P} (\Omega)$. Hence, inequality \eqref{eq:muisminimum} holds with sufficiently small $ \epsilon $  even if  $\mu$ is just a directional local minimizer (and therefore, also if it is a local minimizer in  total variation  or in the Wasserstein distance $W_p$, $1\le p< \infty$), and one arrives at a  contradiction in the same way.

%\ov{and this one}
%Analyzing the proof of Theorem \ref{thm:Constant on Supp}, we find that for
%$\nu$ defined in \eqref{eq:2nu1}, $ \| \varepsilon \nu \|  = \varepsilon
%\|\nu\| \leq 2\epsilon $. Hence, inequality \eqref{eq:muisminimum} holds with sufficiently small $ \epsilon $ when  $\mu$ is only a local minimizer, and one arrives at a contradiction in the same way.

\begin{corollary}\label{cor:Constant on Supp}
The statement of Theorem \ref{thm:Constant on Supp} remains true if we only  assume that $\mu$ is a local (not global) minimizer of $I_K$ (in either of three senses: directional, in total variation, or in Wasserstein $W_p$ metric, $1\le p<\infty$). 
\end{corollary}

As we shall see in Theorem \ref{thm:CPDPotConstMin}, under some additional conditions, in particular, if $K$ is conditionally positive definite, the statement of Theorem \ref{thm:Constant on Supp} can be reversed.

\subsection{Positive Definiteness and   Hilbert--Schmidt Operators.}\label{sec:HS}

Let $\mu$ be a Borel probability measure on $\Omega$ and let $K$ be a continuous function on $\Omega \times \Omega$.  We shall consider the operator $T_{K,\mu}$ associated to $K$ on the space of real-valued functions on $\Omega$ that are square-integrable with respect to $\mu$, $L^2 (\Omega, \mu)$. This is a linear integral operator with kernel $K$ defined by 
\begin{equation}
T_{K,\mu} \psi(x) =  \int\limits_{\Omega}  K (x,y) \psi(y)\, d\mu (y). 
\end{equation}

\begin{lemma}\label{lem:Hilbert-Schmidt}
Let $\widetilde{\Omega} = \operatorname{supp}(\mu)$. The operator $T_{K, \mu}$ is self-adjoint and Hilbert--Schmidt, and the eigenfunctions of $T_{K, \mu}$ corresponding to non-zero eigenvalues are continuous on $\widetilde{\Omega}$. The kernel $K$ is positive definite on $\widetilde{\Omega}$ if and only if $T_{K, \mu}$ is a positive operator on $L^2(\Omega, \mu)$.  
\end{lemma}

\begin{proof}
Self-adjointedness immediately follows from the fact that $K(x,y)$ is symmetric. Since $\Omega$ is compact (and hence $\widetilde{\Omega}$ is also) and $K$ continuous, we know that
$$\int\limits_{\Omega} \int\limits_{\Omega} |K(x,y)|^2\, d \mu(x) d \mu(y) < \infty,$$
which implies that $T_{K,\mu}$ is Hilbert--Schmidt.

Now, suppose that $T_{K, \mu} \phi = \lambda \phi$ for $\lambda \neq 0$. Then the representation
$$ \phi(x) = \frac{1}{\lambda} \int\limits_{\Omega} K(x,y) \phi(y)\, d \mu(y)$$
implies that $\phi$ is continuous on $\widetilde{\Omega}$. 

We now show that the positive definiteness of $K$ and the positivity of $T_{K, \mu}$ are equivalent. If $K$ is positive definite on $\widetilde{\Omega}$, then for any $\psi \in L^2(\Omega, \mu) \subset L^1 (\Omega,\mu)$,
$$ \langle \psi, T_{K, \mu} \psi \rangle_{L^2( \Omega, \mu)} = \int\limits_{\Omega} \int\limits_{\Omega} K(x,y) \psi(x) \psi(y)\, d \mu(x) d \mu(y) = I_K( \psi( \cdot) \mu) \geq 0,$$
so $T_{K, \mu}$ is indeed positive.

Assume instead that $T_{K,\mu}$ is positive. Observe that measures, which are absolutely continuous with respect to $\mu$ and have bounded density, i.e. measures of the form $d\nu = f \, d\mu$, where $f$ is a bounded Borel measurable function on $\widetilde{\Omega}$, are weak$^*$ dense  in $\mathcal M (\widetilde{\Omega})$. To show this, notice that for each ball $B(z,r)$ of radius $r>0$ centered at the point  $z\in \widetilde{\Omega}$, we have $\mu (B(z,r)) \neq 0$, and therefore,   the functions $f_r (x) =  \frac{1}{\mu (B(z,r))} {\mathbbm{1}}_{B(z,r)} (x)$ are well-defined and bounded. Obviously, the measures $\nu_r$ defined by $d\nu_r = f_r\, d\mu$  converge weak$^*$ to $\delta_z$ as $r\rightarrow 0$, which   suffices due to weak$^*$ density of discrete measures. 

Then for all such measures of the form $d\nu  = f \, d\mu$, since bounded functions are in $L^2(\Omega,\mu)$, we have  
$$ I_K (\nu) = \langle T_{K,\mu} f,  f \rangle \ge 0, $$
and by weak$^*$ density, it follows that $K$ is  positive definite on $\widetilde{\Omega}$. 
\end{proof}

Moreover, since $T_{K, \mu}$ is a Hilbert-Schmidt operator, it is in fact a compact operator. Hence, we may apply the Spectral Theorem
%\begin{theorem}[Spectral Theorem for Compact Operators]\label{thm:Spectral}
%Suppose that $H$ is a Hilbert space and $T: H \rightarrow H$ is a compact, self-adjoint operator. Then there exists an orthonormal basis $\{ \phi_j \}_{j=1}^{\dim(H)}$ of $H$  consisting of eigenvectors of $T$ and corresponding eigenvalues $\{ \lambda \}_{j=1}^{\dim(H)}$ such that $|\lambda_j| \geq |\lambda_{j+1}|$ for $1 \leq j < \dim(H)$. If $\dim(H) = \infty$, then $\lim_{j \rightarrow \infty} \lambda_j = 0$. 
%\end{theorem}
to establish that  there exists an orthonormal basis $\{ \phi_j\}_{j=1}^{\dim(L^2(\Omega, \mu))}$ of $L^2 (\Omega, \mu)$ consisting of eigenfunctions of $T_{K, \mu}$, i.e. $T_{K, \mu} \phi_j = \lambda_j \phi_j$, where the sequence of eigenvalues satisfies $|\lambda_j| \geq |\lambda_{j+1}|$ and $\lim_{j \rightarrow \infty} \lambda_j = 0$ if $\dim( L^2(\Omega, \mu)) = \infty$. Moreover, the Spectral Theorem tells us that, for any continuous function $K$, in the $L^2$ sense,
\begin{equation}\label{eq:Mercer0}
K(x,y) = \sum_{j=1}^{\dim(L^2(\Omega, \mu))} \lambda_j \phi_j(x) \phi_j(y).
\end{equation}
%For the rest of this section, we will assume that $\operatorname{supp}(\mu) =  \Omega$, in other words, every non-empty open subset of $\Omega$ has strictly positive measure (if $\operatorname{supp}(\mu) \neq \Omega$, then the results of this section apply for $\widetilde{\Omega} = \operatorname{supp}(\mu)$), and that $K$, $\{ \phi_j \}_{j=1}^{\dim( L^2(\Omega, \mu))}$, and $\{ \lambda \}_{j=1}^{\dim( L^2(\Omega, \mu))}$ are as above. If $K$ is positive definite, we know that $T_{K, \mu}$ is a positive operator, so for each $j \geq 1$,
%$$ \lambda_j = \langle \phi_j, \lambda_j \phi_j \rangle_{L^2(\Omega, \mu)} = \langle \phi_j, T_{K, \mu} \phi_j \rangle_{L^2(\Omega, \mu)} \geq 0.$$
When $K$ is positive definite, $T_{K,\mu}$ is positive, i.e. $\lambda_j \ge 0 $ for all $j\ge 1$.    Mercer's Theorem then strengthens the information about convergence in \eqref{eq:Mercer0}.  We include its proof for completeness. 

\begin{theorem}[Mercer's Theorem]\label{thm:Mercer}
Assume that the kernel $K$  on $\Omega \times \Omega$ is positive definite. Fix a measure $\mu\in \mathbb P(\Omega)$ with  $\operatorname{supp}(\mu) =  \Omega$.   Let $\lambda_j \ge 0$ be the eigenvalues and $\phi_j$ be the eigenfunctions of the associated Hilbert--Schmidt operator $T_{K,\mu}$.  Then
\begin{equation}\label{eq:Mercer}
K(x,y) = \sum_{j=1}^{\dim( L^2( \Omega, \mu))} \lambda_j  \phi_j (x) \phi_j (y),
\end{equation}
where the series converges absolutely and uniformly. 
\end{theorem}

\noindent {\it{Remark:}} Here, as well as in Proposition \ref{prop:Convolution}, we  assume that $\operatorname{supp}(\mu) =  \Omega$. If $\operatorname{supp}(\mu) \neq \Omega$, then the expansion \eqref{eq:Mercer}  holds for  $x,y \in \widetilde{\Omega} = \operatorname{supp}(\mu)$. 

\begin{proof}
As mentioned above, the fact that \eqref{eq:Mercer} holds in the $L^2$ sense follows from the Spectral Theorem for compact operators, hence, only uniform and absolute convergence  in the case $\dim(L^2(\Omega, \mu)) =  \infty$ need to be proven. % They immediately follow if $\dim(L^2(\Omega, \mu)) < \infty$, so we assume that $L^2(\Omega, \mu)$ is infinite dimensional.

Consider the remainder of the series, i.e. the  continuous function 
$$R_N (x,y) = K (x,y) - \sum_{j=1}^N \lambda_j \phi_j (x) \phi_j (y) .$$
The corresponding Hilbert-Schmidt operator defined by $ T_{R_N, \mu} \psi(x) = \int\limits_{\Omega} R_N(x,y)\, \psi(y)\, d\mu (y)$ is clearly bounded and positive: if $ \psi= \sum_{j=1}^\infty \widehat{\psi}_j \phi_j$, then 
$$ \langle T_{R_N, \mu} \psi, \psi \rangle = \sum_{j=N+1}^\infty  \lambda_j  \big| \widehat{\psi}_j \big|^2 \ge 0. $$

For a positive Hilbert--Schmidt operator, the kernel is non-negative on the diagonal, i.e. $R_N (x,x) \ge 0$ for each $x\in \Omega  $.  Indeed, assume that for some $x \in \Omega$ we have $R_N (x,x) <0$. Then we can choose a neighborhood $U$ of $x$ so that $R_N$ is negative on $U\times U$, so
$$ \langle T_{R_N, \mu} \mathbbm{1}_{U}, \mathbbm{1}_{U} \rangle = \int\limits_U \int\limits_U R_N (x,y)\, d\mu(x) d\mu (y) <0, $$
which is a contradiction.

Thus,  $R_N (x,x) \ge 0$ for each $x\in \Omega$, i.e. for any   $N\ge 1$ we have $\displaystyle{\sum_{j=1}^N \lambda_j \phi_j^2 (x) \le K(x,x)}$, and thus 
$$\displaystyle{\sum_{j=1}^\infty \lambda_j \phi_j^2 (x) \le K(x,x)}.$$
Invoking the Cauchy--Schwarz inequality, we find that 
\begin{align*}
\bigg|  \sum_{j=1}^\infty \lambda_j  \phi_j (x) \phi_j (y) \bigg| & \le \sum_{j=1}^\infty \lambda_j  \big| \phi_j (x) \phi_j (y) \big|  \le \bigg(  \sum_{j=1}^\infty \lambda_j \phi^2_j (x)  \bigg)^{1/2} \bigg(  \sum_{j=1}^\infty \lambda_j \phi^2_j (y)  \bigg)^{1/2}\\
& \le K^{1/2} (x,x) K^{1/2} (y,y) \le \sup_{x\in \Omega} K (x,x) < \infty. 
\end{align*}
Therefore, the series in \eqref{eq:Mercer} converges absolutely and uniformly, by Dini's theorem.
\end{proof}

As a simple corollary of Mercer's Theorem, we find that  
$$ \int\limits_\Omega K(x,x)\, d\mu (x)  = \sum_{j\ge 1}   \lambda_j \int\limits \phi_j^2\, d\mu  = \sum_{j\ge 1}   \lambda_j, $$
and thus 
\begin{equation}\label{eq:Mercer Coef Sum}
\sum_{j=1}^{\dim( L^2( \Omega, \mu))} \lambda_j < \infty.
\end{equation}

By combining Mercer's Theorem with Lemma \ref{lem:PDconstructed} we arrive at a characterization of all positive definite kernels.
\begin{corollary}\label{cor:PDCharacterization}
The kernel $K$ on $\Omega \times \Omega$  is positive definite if and only if for some sequence of   functions $\phi_j: \Omega \rightarrow \mathbb R$ and real numbers $\lambda_j \ge 0$, the kernel $K$ %orthonormal basis $\{ \psi_j \}_{j=1}^{\dim(L^2(\Omega, \mu))}$ of $L^2(\Omega, \mu)$ and sequence of nonnegative real numbers $\{\kappa_j \}_{j=1}^{\dim(L^2(\Omega, \mu))}$,
$$ K(x,y) = \sum_{j\ge 1}  \lambda_j \phi_j(x) \phi_j(y)$$
where the sum converges absolutely and uniformly. % $\kappa_j \geq 0$ for each $j$, and $\psi_j$ is continuous whenever $\kappa_j > 0$.
\end{corollary}

\subsection{The Existence of Convolution Square Root}\label{sec:SqRoot}
 
 We now  supply another way to characterize positive definite functions, which is similar to the existence of the ``square root'' for positive semidefinite matrices. A similar statement for the rotation-invariant kernels on the sphere has been obtained in \cites{BD, BDM}. Here we prove a general version of this representation.

\begin{proposition}\label{prop:Convolution} Fix any $\mu\in  \mathbb P (\Omega)$ with $\operatorname{supp}(\mu) =  \Omega$.
A kernel $K$ on $\Omega \times \Omega$  is positive definite if and only if there exists some $k \in L^2( \Omega \times \Omega, \mu \times \mu)$ such that % $k$ can be represented as
%$$k(x,y) = \sum_{j=1}^{\dim(L^2(\Omega, \mu))} \kappa_j \phi_j(x) \phi_j,$$
%in the $L^2$ sense, and
for all $x, y \in \Omega$,
\begin{equation}\label{eq:Convolution}
K(x,y) = \int\limits_{\Omega} k(x, z) k(z,y)\, d \mu(z).
\end{equation}
\end{proposition}

We again  remind the reader  that if $\widetilde\Omega = \operatorname{supp}(\mu)  \subsetneq \Omega$, then the conclusion of Proposition \ref{prop:Convolution} holds on $\widetilde\Omega$. 

%As will be shown in the proof below, the coefficients $\kappa_j$ satisfy $\kappa_j^2 = \lambda_j$ for all $j$.

\begin{proof} 
Assume that $K$ is positive definite and define the function $k : \Omega \times \Omega \rightarrow \mathbb R$ by setting
\begin{equation}\label{eq:k}
k (x,y) = \sum_{j\ge 1 }  \sqrt{\lambda_j}\, \phi_j (x) \phi_j (y),
\end{equation}
where $\lambda_j \ge 0$ are eigenvalues and $\phi_j$ are eigenfunctions of $T_{K,\mu}$. 
The sequence $\{ \phi_j (x) \phi_j (y) \}$ is orthonormal  in $ L^2 (\Omega\times \Omega, \mu \times \mu)$,  hence  $k \in L^2 (\Omega\times \Omega, \mu \times \mu)$, since $$ \| k \|^2_{ L^2 (\Omega\times \Omega, \mu \times \mu)} =  \sum_{j\ge 1} \lambda_j   <  \infty$$   according to \eqref{eq:Mercer Coef Sum}.

Moreover, for each $x\in \Omega$, we have that $k (x, y) \in L^2 (\Omega, \mu)$ as a function of $y$. Indeed, using Mercer's theorem and applying Plancherel's theorem with respect to $d\mu (y)$ to \eqref{eq:k}, we can compute the $L^2 $ norm
$$ \| k (x, \cdot ) \|_{L^2 (\Omega, \mu)}^2 = \sum_{j\geq1} \big| \sqrt{\lambda_j} \phi_j (x) \big| ^2  = \sum_{j\geq1} \lambda_j \phi^2_j(x) = K(x,x) < \infty. 
$$

%For any choices of real $\kappa_j$'s, we have, in the $L^2$ sense, the following equalities: 
Similarly, since $k \in L^2 (\Omega,\mu)$ in each variable,  we have, for any $x,y \in \Omega$,  
$$ \int\limits_{\Omega} k(x, z) k(z,y)\, d \mu(z) = \langle k(x, \cdot), k(y, \cdot) \rangle_{L^2(\Omega, \mu)} = \sum_{j\ge 1} \sqrt{\lambda_j}  \phi_j(x) \sqrt{\lambda_j} \phi_j(y) = K(x,y),$$ due to Mercer's Theorem, which proves \eqref{eq:Convolution}. \\

Alternatively, if we assume that \eqref{eq:Convolution} holds, 
%$$K(x,y) = \int\limits_{\Omega} k(x, z) k(z,y) d \mu(z) ,$$
then for any finite point configuration $\omega_N = \{ z_1, ..., z_N\}$ in $\Omega$ and $c_1, ..., c_N \in \mathbb{R}$, we have
\begin{align*}
\sum_{j=1}^{N} \sum_{i = 1}^{N} K(z_j, z_i) c_i c_j & = \sum_{j=1}^{N} \sum_{i = 1}^{N} c_j c_i \int\limits_{\Omega} k(z_j, z) k(z,z_i)\, d \mu(z) \\
& = \int\limits_{\Omega} \Big( \sum_{j=1}^{N} c_j k( z_j, z) \Big)^2 d \mu(z) \geq 0.
\end{align*}
Thus, $K$ is clearly positive definite.
\end{proof}

\noindent {\it{Remark: }} Observe that the choice of the ``square root'' $k$ is not unique. Indeed, instead of \eqref{eq:k} one could take any $$k(x,y) = \sum_{j\ge 1}  \kappa_j \phi_j(x) \phi_j(y) ,$$ with the property  that $\kappa_j^2 = \lambda_j$ for all $j\ge 1$, i.e. $\kappa_j = \pm \sqrt{\lambda_j}$ for an arbitrary choice of signs, yielding uncountably many functions $k$ satisfying \eqref{eq:Convolution}. \\

\subsection{Energy Minimizers and Hilbert--Schmidt Operators}\label{sec:MinHS}

There is a close relation between energy minimizers and the properties of the associated Hilbert--Schmidt operator $T_{K,\mu}$ on $L^2 (\Omega,\mu)$. We have the following statement.

\begin{lemma}\label{lem:OperatorPos}
Let $K$ be a kernel on $\Omega \times \Omega$ and assume that $\mu \in \mathbb{P}(\Omega)$ is  a global or local minimizer of $I_K$ with $I_K(\mu)  \geq 0$. Then the Hilbert--Schmidt operator  $T_{K, \mu}$ is positive.
\end{lemma}

\begin{proof}
We start by observing that if $\mu$ is a (global or local) minimizer of $I_K$, then the constant function $\mathbbm{1}_{\Omega}$ is an eigenfunction of $T_{K,\mu}$ in $L^2(\Omega, \mu)$. Indeed, according to  Theorem \ref{thm:Constant on Supp} or Corollary \ref{cor:Constant on Supp}, for each $x \in \operatorname{supp} (\mu)$, 
\begin{equation}\label{eq:oneeigen}
T_{K,\mu} \mathbbm{1}_{\Omega} (x) = \int\limits_\Omega K(x,y)\, d\mu (y)  = U_K^\mu (x)  = I_K (\mu) \mathbbm{1}_{\Omega} (x). 
\end{equation}

Assume, indirectly, that $T_{K, \mu}$ is not positive. By Lemma \ref{lem:Hilbert-Schmidt},   $T_{K, \mu}$ is compact and self-adjoint, so there exists an eigenfunction $\phi$ such that $T_{K, \mu} \phi = \lambda \phi$ with  $\lambda < 0$. Since $\phi$ is continuous, and therefore bounded, on $\operatorname{supp}(\mu)$, we have that for sufficiently small $t>0$, the measure
$$ \mu_t = (1 + t \phi) \mu$$
is positive. As we noted above, $\mathbbm{1}_{\Omega}$ is an eigenfunction of $T_{K, \mu}$ corresponding to the eigenvalue $I_K(\mu) \geq 0$. Clearly, then, $\mathbbm{1}_{\Omega}$ and $\phi$ are orthogonal, so
$$ \mu_t( \Omega) = \int\limits_{\Omega} (1 + t \phi(x))\, d \mu(x) = \mu( \Omega) = 1,$$
and % \ov{here replaced linear combination with $ \mu_t $}
\begin{align*}
I_K(\mu_t) &  = \int\limits_{\Omega} \int\limits_{\Omega} K(x,y) (1 + t \phi(x)) \big(1 + t \phi(y) \big)\,  d \mu(x)   d \mu(y)\\
&  = I_K(\mu) + \lambda t^2 \int\limits_{\Omega} |  \phi(x) |^2   d \mu(x) < I_K(\mu),
\end{align*}
which contradicts the (local) minimality of $\mu$ over probability measures.  
\end{proof}

Recall that according to  Lemma \ref{lem:Hilbert-Schmidt}, the operator  $T_{K,\mu}$ is positive if and only if $K$ is positive definite on the support of $\mu$.  If the condition $I_K (\mu) \ge 0$ in Lemma \ref{lem:OperatorPos} is not satisfied, we can replace $K$ by $K'(x,y) = K(x,y) - I_K(\mu)$, which does not affect energy  minimizers. Then $I_{K'}(\mu) = 0$, and hence,  Lemma \ref{lem:OperatorPos} applies.  Therefore $T_{K',\mu}$ is positive, i.e. $K'$ is positive definite on $\widetilde{\Omega} = \operatorname{supp}(\mu)$. In other words, $K$ is positive definite up to an additive constant, as a kernel on $\widetilde{\Omega}  \times \widetilde{\Omega}$. We arrive at the following important fact.

\begin{lemma}\label{lem:PosDefonSupp}
    Let $K$ be a kernel on $\Omega\times \Omega $. Suppose that $\mu$ is a
    {local} minimizer of $I_{K}$ over $\mathbb{P}( \Omega)$. Then the kernel
    $K$ must be positive definite modulo a constant on
    $\operatorname{supp}(\mu)$, i.e. as a kernel on $\operatorname{supp}(\mu)
    \times \operatorname{supp}(\mu)$. If $I_K(\mu) \geq 0$, then $K$ is
    positive definite on $\operatorname{supp}(\mu)$.
\end{lemma}

Various statements of this type are known in the literature \cites{CFP, FS}.  Lemma \ref{lem:PosDefonSupp} clearly implies the following localization statement:
\begin{corollary}\label{cor:NotFullSupp}
Assume the the kernel $K$ on $\Omega \times \Omega$ is not positive definite up to an additive constant. Then any (local or global) minimizer $\mu$ of $I_K$ must be supported on a proper subset of $\Omega$, i.e. $\operatorname{supp} (\mu) \subsetneq \Omega$. 
\end{corollary}

\section{Invariant Measures}\label{sec:InvMeasEnergy}

As suggested in Section \ref{sec:MinMeas}, measures with constant potentials are particularly interesting from the point of view of energy minimization. They also naturally arise in metric geometry, in connection with the so-called ``rendezvous numbers" \cite{CMY}, and we draw the term ``invariant'' from this literature. These applications and various interesting properties warrant a separate discussion of such measures.

\subsection{Definition, Examples, and Comments} 
We start with the following definition.

\begin{definition}\label{def:Kinv}
    Let $K$ be a kernel on $\Omega \times \Omega$. We say that a measure $\mu \in \mathbb{P}(\Omega)$ is \textbf{$K$-invariant}  on $\Omega$ %\ov{\ul{on $ \Omega $?}} 
    if the potential of this measure with respect to $K$  is constant on $\Omega$, i.e.
\begin{equation}\label{eq:Kinv}
U_K^\mu ( x )  = I_K (\mu) \,\,\, \textup{ for every } \,\, x\in \Omega. 
\end{equation}
\end{definition}
We shall see shortly that these measures have an array of remarkable properties. Notice that the definition does not require that $\mu$ has full support: the majority of statements in this section hold even in the absence of this assumption. 
 %while some of the statements in this chapter  will require this additional assumption, i.e. $\operatorname{supp} (\mu) = \Omega$, the majority of them hold even in its absence. 

Before proceeding to these properties we shall provide  some examples, showing that $K$-invariance  is a rather rich notion. Observe that  most of the examples below have full support. All statements that we shall prove for $K$-invariant measures will apply, in particular,  to  these natural examples. 

\begin{itemize}
\item If $\mu$ (locally)  minimizes $I_K$ and has full support, then according to Theorem \ref{thm:Constant on Supp}  and Corollary \ref{cor:Constant on Supp}   , the measure $\mu$ is $K$-invariant. 
\item Let $\Omega = \mathbb S^{d-1}$ and assume that $K$ is rotationally invariant, i.e. $K(x,y) = F (\langle x,y\rangle)$. Then the normalized uniform surface measure is $K$-invariant, since the potential $U_F^\sigma (x) = \int\limits_{\mathbb S^{d-1}} F (\langle x,y\rangle)\, d\sigma (y)$ is obviously independent of $x \in \mathbb S^{d-1}$. 
\item If, moreover, the function $F$ from the previous example is a polynomial of degree $M$, and $\omega_N = \{ z_1,...,z_N \} \subset \mathbb S^{d-1}$  is a spherical $M$-design, then $\mu = \frac1{N} \sum_{i=1}^N \delta_{z_i} $ is also  $K$-invariant.
\item Similarly, assume that $\Omega$ is a compact  two-point homogeneous space, $K$ is invariant with respect to the group of isometries, and  $\eta$ is the normalized uniform measure on $\Omega$. Then $\eta$ is $K$-invariant. This equally applies to connected (e.g., projective spaces) and discrete (e.g., Hamming cube) two-point homogeneous spaces.  
\item If $\Omega$  is a compact topological group, $\mu$ is its normalized Haar measure, and $K$ is invariant with respect to the group operation, i.e. $K(x,y) = F(y^{-1}x)$, then $\mu$ is $K$-invariant. 
\item A pair $ K $, $ \mu $ as in the above example can be constructed as follows. Suppose $ G $ is a compact group acting on $ \Omega $ transitively, $ \nu $ is the Haar measure on $ G $, and $ \mu $ -- a $ G $-invariant measure on $ \Omega $. Then for any function $ f \in L^2(\Omega, \mu) $, the kernel
    \[
        K(x,y) =  \int_{G} f(\tau x) f(\tau y) \, d\nu(\tau)
    \]
    is positive definite and $ \mu $ is $ K $-invariant. Indeed, $ K $ is an average of positive definite functions $ f(\tau x) f(\tau y)$; also, $ U^\mu_K $ is a $ G $-invariant function on $ \Omega $ and, since $ G $ acts transitively, must be constant there (see e.g. \cite{CSFSV} for further properties of such $ K $).

\item Let $\Omega = [-1,1]$ and $K (x,y ) = |x - y|$. Then the measure $\mu = \frac12 (\delta_0 + \delta_1) $ is $K$-invariant. Notice that this example (as well as spherical designs) provides an invariant measure which  does not have full support. 
\item More generally, if $\Omega = \mathbb B^d$ is the closed unit ball in $\mathbb R^d$ and $K(x,y) = \|x-y\|^{-(d-2)}$ (or $K(x,y)= - \log \| x-y \| $ for $d=2$)  is the Newtonian potential, then the uniform surface measure $\sigma$ on $\mathbb S^{d-1}$ is the equilibrium measure for $I_K$ and is  $K$-invariant on $\mathbb B^d$ \cite{BHS,L}, which is known as the Faraday cage effect. For $d=1$, one gets exactly the previous example, while for $d\ge 2$, the kernels are discontinuous (and thus are  partially  beyond the scope of our discussion). 
\item Still more general form of the Faraday's effect applies to $ \Omega \subset \mathbb R^3 $ that is given by a finite union of disjoint closed domains with smooth boundaries,  for $ K(x,y) = \|x-y\|^{-s} $, $ 1\leq s < 3 $: there exists a probability measure with potential constant on $ \Omega $ \cite[Section 17]{Fr}.
\end{itemize}
%About that reference: equilibrium measures for the unit ball are computed in Prop. 4.6.4 of Borodachov-Hardin-Saff (p 180). The fact that potential is constant inside the ball follows from subharmonicity for d-2 ≤ s < d, and can be seen as equation (1.3.11) in Landkof (p 67).

Despite an abundance of examples, the existence of a $K$-invariant measure is possible only under  certain restrictions on the geometry of the domain $\Omega$ and the structure of the kernel $K$.  For example, the following statement is true \cite{CMY}.

%\ov{If $ U $ is a vector space over $ \mathbb R $, can we just write $ \mathbb R^d $ in the lemma?}
\begin{lemma}\label{lem:3pts}
Assume that $\mathbb R^d$ is  endowed %a finite-dimensional vector space endowed 
with a  strictly convex norm $|\cdot |$, i.e. $|x+y| < |x| + | y|$ unless $x$ and $y \in \mathbb R^d$ have the same direction. 
Let $\Omega \subset \mathbb R^d $ be compact and  set $K(x,y) =| x-y |$.  If there exists a $K$-invariant measure on $\Omega$, then either $\Omega$ is a line segment or no three points of $\Omega$ are collinear. 
\end{lemma}
For the case when $K(x,y) =\|x-y\|$ is the Euclidean distance, this lemma shows, for example, that an invariant measure doesn't exist for the unit ball, while, as we know, it does exist for the sphere.

Finally, we make the remark that if a measure $\mu$ is $K$-invariant, it implies that a constant function is an eigenfunction of the Hilbert--Schmidt operator $T_{K,\mu}$ in $L^2 (\Omega,\mu)$ with eigenvalue $\lambda = I_K ( \mu)$, which is implied by \eqref{eq:oneeigen}.

\subsection{A Crucial Identity}

The following simple relation provides a powerful direct link between energy minimization and (conditional) positive definiteness and will play a decisive role in many results of this section. It is also an important first step in the proof of the Generalized Stolarsky principle (Theorem \ref{thm:genStol}). In a nutshell, this lemma states that, while $I_K$ is a quadratic functional, it behaves linearly around a $K$-invariant measure. 

\begin{lemma}\label{lem:NuMinusMu}
    Let $K$ be a kernel on $\Omega \times \Omega$ and let $\mu$ be a $K$-invariant measure, i.e. $U_K^{\mu}(x) = I_K (\mu) $ for  all  $x\in \Omega$. Then for any $\nu \in \widetilde{\mathbb{P}} (\Omega)$,
\begin{equation}\label{eq:Linear1}
I_K(\nu - \mu) = I_K(\nu) - I_K(\mu).
\end{equation}

\noindent More generally, if $\mu \in \mathbb{P}(\Omega)$ satisfies  $U_K^{\mu}(x) \geq I_K(\mu)$, with equality on $\operatorname{supp}(\mu)$, then for any $\nu \in {\mathbb{P}}(\Omega)$
\begin{equation}\label{eq:subLinear}
I_K(\nu - \mu) \leq I_K(\nu) - I_K(\mu),
\end{equation}
and equality \eqref{eq:Linear1} holds for any measure $\nu  \in \widetilde{\mathbb{P}}(\Omega)$ with $\operatorname{supp}(\nu) \subseteq \operatorname{supp}(\mu)$. 
\end{lemma}

\begin{proof}
If $\mu$ is $K$-invariant, then for any $\nu  \in \widetilde{\mathbb{P}}(\Omega)$,
\begin{equation}\label{eq:MixedInvariant}
I_K (\mu, \nu) = \int\limits_{\Omega} U_K^{\mu}(x)\, d\nu(x) = \int\limits_{\Omega} I_K(\mu)\, d \nu(x) = I_K (\mu).
\end{equation}
Therefore
\begin{align*}
I_K( \nu - \mu) & = I_K(\nu) - 2 I_K (\mu, \nu) %\int\limits_{\Omega} U_K^{\mu}(x) d\nu(x) 
+ I_K(\mu)  = I_K (\nu) - I_K (\mu).
\end{align*}

For the second part of our claim, observe  that for any $ \nu \in {\mathbb{P}}(\Omega)$, instead of equality \eqref{eq:MixedInvariant}, one has the inequality $I_K (\mu,\nu) \ge I_K (\mu)$, and thus, 
\begin{equation*}
I_K(\nu - \mu) = I_K(\nu) - 2 I_K(\mu,\nu) + I_K(\mu) \le I_K(\nu) - I_K(\mu).
\end{equation*}
Finally, the last statement follows from the first  one by replacing $\Omega$ with $\operatorname{supp} (\mu)$. 
\end{proof}

Theorem \ref{thm:Constant on Supp} and Corollary \ref{cor:Constant on Supp}  show that if  $\mu$ is a global (or at least local) minimizer of $I_K$, it satisfies the  the conditions of  the second statement in Lemma \ref{lem:NuMinusMu}, and if in addition $\mu$ has full support, it also satisfies the first condition, i.e. $\mu$ is $K$-invariant. Thus Lemma \ref{lem:NuMinusMu} applies to (local) energy minimizers, which  results in the following corollary:

\begin{corollary}\label{prop:BoundonNuMinusMu}
Let $K$ be a kernel on $\Omega \times \Omega$ and $\mu$ be a (local) minimizer of $I_K$. Then for any $\nu \in \mathbb{P}(\Omega)$
\begin{equation}\label{eq:LMsublinear}
I_K(\nu - \mu) \leq I_K(\nu) - I_K(\mu).
\end{equation}
For any $\nu \in \widetilde{\mathbb{P}}(\Omega)$ such that $\operatorname{supp}(\nu) \subseteq \operatorname{supp} (\mu)$, then
\begin{equation}
I_K(\nu - \mu) = I_K(\nu) - I_K(\mu).
\end{equation}
\end{corollary}

\subsection{Conditional Positive Definiteness and Energy Minimization}

Identity \eqref{eq:Linear1} of Lemma \ref{lem:NuMinusMu} provides a  link between energy minimization and conditional positive definiteness. 
We would like to emphasize that   relation \eqref{eq:Linear1} holds not just for probability measures $\nu$, but for arbitrary signed measures of total mass one, i.e. $\nu \in  \widetilde{\mathbb{P}}(\Omega)$. Therefore, one can immediately deduce the following equivalence. 

\begin{theorem}\label{thm:CPDequivMIN}
Let $K$ be a kernel on $\Omega \times \Omega$ and assume that $\mu$ is $K$-invariant. Then $\mu$ minimizes  $I_K$ over the set $\widetilde{\mathbb{P}}(\Omega)$ of normalized signed Borel measures  if and only if $K$ is conditionally positive definite.  

Moreover, $\mu$ uniquely  minimizes  $I_K$ over   $\widetilde{\mathbb{P}}(\Omega)$  if and only if $K$ is conditionally  strictly positive definite.
\end{theorem}

\begin{proof}
Suppose that $K$ is conditionally positive definite. Then for any $\nu \in  \widetilde{\mathbb{P}}(\Omega)$, equality  \eqref{eq:Linear1} holds and, since $(\nu - \mu ) (\Omega ) = 0$, we have 
\begin{equation*}
I_K (\nu ) - I_K (\mu)  = I_K (\nu - \mu ) \ge 0,
\end{equation*} 
which shows that $\mu$ minimizes $I_K$ over $ \widetilde{\mathbb{P}}(\Omega)$. If $K$ is conditionally strictly positive definite, then $I_K (\nu ) = I_K (\mu)$ only if $\nu - \mu =0$,  %  on all Borel sets of $\Omega$, 
i.e. $\mu$ is the unique minimizer. 

Assume conversely that $I_K (\mu ) \le I_K (\nu)$ for each $\nu \in  \widetilde{\mathbb{P}}(\Omega)$. Consider an arbitrary signed measure $\gamma \in \mathcal Z (\Omega)$. Define $\nu = \mu +\gamma$, then $\nu (\Omega ) = 1$, i.e. $\nu \in  \widetilde{\mathbb{P}}(\Omega)$. Thus, applying \eqref{eq:Linear1} once again, we find that 
\begin{equation*}
I_K (\gamma ) = I_K (\nu - \mu ) = I_K (\nu ) - I_K (\mu) \ge 0,
\end{equation*}
hence $K$ is conditionally positive definite. If $\mu$ is the unique minimizer, then the expression above equals zero only for $\gamma=0$, i.e. $K$ is conditionally strictly positive definite. 
\end{proof}

% Under the assumptions of Theorem~\ref{thm:CPDEquivalences}, we have a very general analog of the linear programming bound, using $ \phi_j $ in place of spherical harmonics; compare it to \cite[Theorem 5, Chapter 9.3]{CS}. On 2-point homogeneous spaces, $ \phi_j $ are the eigenfunctions of the Laplacian (spherical harmonics on the sphere), and $ K $ is assumed to be rotationally invariant, see Section~\ref{sec:Sphere}.
% \begin{corollary}
%     Suppose $ K $ on $ \Omega\times \Omega $ satisfies
%     $$ K(x,y) = \sum_{j\ge 1}  \lambda_j \phi_j(x) \phi_j(y)$$
%     with $ \phi_1 \equiv const $, $\lambda_j\geq 0$. Then for any collection $ \{ z_i \}_{i=1}^N \subset \Omega $ there holds
%     \[
%         \frac 1{N^2}\sum_{i,l =1}^N K(z_i, z_l)  \geq \phi_1.
%     \]
% \end{corollary}

Obviously,  one of the implications holds for minimizers over probability measures $\mathbb{P}(\Omega)$. 

\begin{corollary}\label{cor:CPDequivMIN}
Let $K$ be a kernel on $\Omega \times \Omega$ and assume that $\mu$ is $K$-invariant.  If  $K$ is conditionally (strictly) positive definite, then $\mu$ (uniquely) minimizes  $I_K$ over the set ${\mathbb{P}}(\Omega)$ of   Borel probability measures. 
\end{corollary}
This corollary is well known, see, e.g., Theorem 4.2.11 in \cite{BHS}, but the equivalence in Theorem \ref{thm:CPDequivMIN} appears to be new.

Some remarks are in order. We would like to remind the reader that in some specific cases, such as the sphere with the uniform surface measure (or more generally, two-point homogeneous spaces with the corresponding uniform measures),  the relation between energy minimization and some form of positive definiteness of the kernel is well known \cites{Sc, BDM}. However, it is usually demonstrated using the representation theory of the underlying space and the associated orthogonal polynomial (Gegenbauer, Jacobi, Krawtchouk) expansions. In fact, Theorem~\ref{thm:CPDequivMIN} can be viewed as a generalization of the so-called {\it mean inequality} due to Kabatianskii and Levenshtein \cite[Theorem 5, Chapter 9.3]{CS}, in which $ K $ is an invariant positive definite kernel on a 2-point homogeneous space. The Theorem is thus a blanket statement that covers all of these examples and beyond. Moreover, it relies only on the completely elementary identity \eqref{eq:Linear1}, thus simplifying the known proofs in all of the specific cases. In the spherical case, a similar approach has been recently employed in \cite{BDM}.

\subsection{Conditional Positive Definiteness and  Positive Definiteness up to a Constant (Revisited)}\label{sec:CPDDPC2}
As we have observed in the previous discussions, two properties, which are somewhat weaker than positive definiteness, play an important role in energy minimization: namely, conditional positive definiteness and positive definiteness up to an additive constant.  We have already demonstrated in Lemma \ref {lem:PD-CPD} that the latter always  implies the former, and the converse implication is not true in general. We shall now show that the converse implication also holds, i.e. conditional positive definiteness implies positive definiteness up to an additive constant, if we additionally assume the existence of a $K$-invariant measure. Moreover, the statement also holds for the ``strict'' version of these properties. 

\begin{lemma}\label{lem:CPDtoPDC}
Let $K$ be a kernel on $\Omega\times \Omega$ and assume that $K$ is conditionally  (strictly) positive definite. Suppose also that there exists a $K$-invariant measure $\mu \in \mathbb{P}(\Omega)$. Then $K$  is (strictly) positive definite up to  an additive constant. 
\end{lemma}

\begin{proof}
Let $K$ be conditionally (strictly) positive definite. Set $C = - I_K (\mu)+ 1$. Then $K+C $ is still conditionally (strictly) positive definite, $\mu$ is $(K+C)$-invariant,  and $I_{K+C}  (\mu ) = (\mu(\Omega))^2 =1 $.  For any signed measure $\nu$ with  $\nu (\Omega ) = 1$, identity \eqref{eq:Linear1} implies that 
\begin{equation*}
 I_{K+C} (\nu) - I_{K+C} (\mu) = I_{K+C} (\nu - \mu) = I_K (\nu- \mu) \ge 0.
\end{equation*}
Therefore, $I_{K+C} (\nu) \ge 1 >0$. 

Now consider an arbitrary measure $\gamma \in \mathcal M (\Omega)$. If $\gamma (\Omega ) = 0$, then $I_{K+C} (\gamma) \ge 0$ by conditional positive definiteness  (and  $I_{K+C} (\gamma) > 0$ for $\gamma \neq 0$ for the ``strict'' version if $\gamma \neq 0$). 

If  $ \gamma (\Omega ) = c \neq 0 $, we can write  $\gamma  = c \nu$ for some  $\nu \in  \widetilde{\mathbb{P}}(\Omega)$. Therefore,  $I_{K+C} (\gamma) = c^2 I_{K+C} (\nu ) \ge c^2 >0$. Hence $K+C$ is  (strictly) positive definite. 
\end{proof}

Lemmas \ref {lem:PD-CPD} and \ref{lem:CPDtoPDC} together   show that  in the presence of a $K$-invariant measure, conditional positive definiteness and positive definiteness modulo an additive constant are equivalent notions. This is the case, for example, for rotationally invariant kernels $K$ on the sphere, since the uniform surface measure $\sigma$ is $K$-invariant for all such kernels.

\subsection{Local and Global Minimizers}\label{sec:LocGlob}
Under some additional assumptions, local minimizers of $I_K$ are necessarily global minimizers. Some facts  of this type have been observed in the  papers of  the authors with D.~ Ferizovi\'c, A.~Glazyrin, J.~Park \cites{BFGMPV, BGMPV}.  Here we prove a variety of more general statements with the same flavor.

In this subsection, the words ``local minimizer''  in the assumptions mean  ``directional local minimizer'' in the sense of Definition \ref{def:locmin}. Thus these results are also valid for local minimizers in  total variation  or in the Wasserstein distance $W_p$ with $1\le p< \infty$, according to \eqref{eq:LocMinRel}.

\begin{proposition}\label{prop:LocalMinGlobalMin}\label{prop:CPDLocalMinGlobalMin} 
Suppose that $\mu$ is a local minimizer of $I_K$ and any of the following two conditions holds:
\begin{enumerate}
\item\label{lg2} the measure $\mu$ is $K$-invariant, i.e. $U_K^{\mu}(x) = I_K (\mu) $ for  all  $x\in \Omega$;
\item\label{lg3}  $K$ is conditionally positive definite. 
\end{enumerate}
Then $\mu$ is a global minimizer of $I_K$.
\end{proposition}

Recall that, according to Corollary \ref{cor:Constant on Supp}, condition \eqref{lg2} is automatically satisfied  if $\mu$ has full support (and so is condition \eqref{lg3}, due to Lemmas \ref{lem:PD-CPD} and  \ref{lem:PosDefonSupp}). Notice also that, unlike Corollary \ref{cor:CPDequivMIN}, part \eqref{lg3} does not require $K$-invariance of $\mu$, but assumes instead that $\mu$ is a local minimizer.

\begin{proof}
 %   \ov{some wording changes in this proof}
    Assume that \eqref{lg2} holds. Then for any $\nu \in  \mathbb{P}(\Omega)$, since %$ \| \mu - ((1-t) \mu + t \nu) \| = t \| \mu -\nu \| \leq 2t $ and 
    %we have $I_K (\mu, \nu) = I_K (\mu)$ and therefore, 
    $\mu$ is a local minimizer, applying equation \eqref{eq:Linear1} of Lemma \ref{lem:NuMinusMu}  for small $t>0$ gives
    \begin{align*}
        0 \le I_K( (1-t) \mu + t \nu) - I_K(\mu) & = I_K \big(  t (\nu -\mu ) \big) = t^2 I_K (\nu - \mu ) = t^2 \Big(I_K (\nu) - I_K (\mu) \Big).
    \end{align*}
    %\begin{align*}
    %I_K(\mu) &\leq I_K( (1-t) \mu + t \nu)
    %= (1-t)^2 I_K(\mu) + 2 t (1-t) I_K(\mu, \nu) + t^2 I_K(\nu) \\
    %& = (1-t)^2 I_K(\mu) + 2 t (1-t) I_K(\mu) + t^2 I_K(\nu) = (1-t^2) I_K(\mu) + t^2 I_K(\nu).
    %\end{align*}
    Thus  $I_K(\mu) \leq I_K(\nu)$, i.e. $\mu$ is a global minimizer of $I_K$ in $ \mathbb{P}(\Omega)$.

    Now assume that   \eqref{lg3} holds. Then  for each $\nu \in  \mathbb{P}(\Omega)$, the measure $\nu-\mu$ has total mass zero. According to conditional positive definiteness of $K$ and  inequality \eqref{eq:LMsublinear} of Corollary \ref{prop:BoundonNuMinusMu},
    \begin{equation}
        I_K(\nu) - I_K(\mu) \geq I_K(\nu - \mu) \ge 0.
    \end{equation}
    Therefore, in this case, $\mu$ also minimizes $I_K$.  
\end{proof}

Observe that if $\mu$ is $K$-invariant, then, according to part \eqref{lg2} of Proposition \ref{prop:LocalMinGlobalMin}, the implications of \eqref{eq:LocMinRel} may be reversed, i.e.
 three definitions of local minimizers (directional, in  total variation, and  in the Wasserstein distance $W_p$, $1\le p< \infty$) are equivalent, since in either of the cases the minimizer is necessarily global. 

\begin{corollary}\label{cor:LocMinEquiv}
Assume that $\mu \in \mathbb P (\Omega)$ is $K$-invariant. Then the following conditions are equivalent:
\begin{itemize}
\item $\mu$ is a directional local minimizer of $I_K$;
\item $\mu$ is local minimizer of $I_K$ with respect to the total variation norm;
\item $\mu$ is local minimizer of $I_K$ with respect to the Wasserstein distance $W_p$, $1\le p< \infty$ (if $\Omega$ is a metric space);
\item $\mu$ is a global minimizer of $I_K$ on $\mathbb P(\Omega)$. 
\end{itemize}
\end{corollary}

When $\mu$ is a local minimizer with full support,  both conditions \eqref{lg2} and \eqref{lg3} of Proposition \ref{prop:LocalMinGlobalMin} hold simultaneously, and an even stronger conclusion can be drawn.

\begin{proposition} \label{prop:CPDSupportLocalMinGlobalMin}
Let $\mu$ be a local minimizer of $I_K$ with $\operatorname{supp} (\mu) = \Omega$. %Assume in addition that $\mu$ is $K$-invariant and $K$ is conditionally positive definite. 
Then $\mu$ is a global minimizer of $I_K$ over $\widetilde{\mathbb{P}}(\Omega)$, the set of all signed Borel measures with total mass one.
\end{proposition} 

\begin{proof} As discussed above, if $\mu$ is a local minimizer with full support, then both $K$-invariance of $\mu$ and conditional positive definiteness of $K$ immediately  follow,  from Corollary \ref{cor:Constant on Supp} and Lemmas \ref{lem:PD-CPD} and  \ref{lem:PosDefonSupp}, respectively. Theorem \ref{thm:CPDequivMIN} then shows that $\mu$ minimizes $I_K$ over $\widetilde{\mathbb P} (\Omega)$. 
%%Let $\nu \in \widetilde{\mathbb{P}}(\Omega)$. Since  $\mu$ is a $K$-invariant local minimizer, we can  apply conditional positive definiteness of $K$ in conjunction with  equality \eqref{eq:Linear1} of Lemma \ref{lem:NuMinusMu} to obtain
%\begin{equation*}
%I_K(\nu) - I_K(\mu) = I_K( \nu - \mu)  \ge 0, 
%\end{equation*}
%giving us our claim.
\end{proof}

Finally, another version of the local-to-global minimization principle may be proved under the assumption that $I_K$ is convex at $\mu$, which according to Lemma  \ref{lem:ConvexEqual2}, is equivalent to the fact that the arithmetic mean inequality \eqref{eq:AM2} holds for any measure $\nu \in \mathbb{P}(\Omega)$.  We have the following statement.

\begin{theorem}\label{thm:CPDPotConstMin}
Suppose that $K$ is a kernel on $\Omega \times \Omega$ and that for some $\mu \in \mathbb{P}(\Omega)$  there exists a constant $M \in \mathbb{R}$ such that $U_K^{\mu}(x) \geq M$, with equality on $\operatorname{supp}(\mu)$. If $I_K$ is convex at $\mu$, then $\mu$ is a global  minimizer of $I_K$ over $\mathbb{P}(\Omega)$. 
\end{theorem}

Before proving this statement we observe that its first assumption is  satisfied in either of the  following two cases:  (a) if $\mu$ is a local minimizer, according to Corollary \ref{cor:Constant on Supp}; (b) if $\mu$ is $K$-invariant.  
Thus we have  two immediate corollaries. The first one is the aforementioned local-to-global principle.

\begin{corollary}\label{cor:LocGlob3}
Let $K$ be a kernel on $\Omega \times \Omega$ and let $\mu \in \mathbb{P}(\Omega)$ be a local minimizer of $I_K$.  If $I_K$ is convex at $\mu$, then $\mu$ is a global minimizer of $I_K$ over $\mathbb{P}(\Omega)$. 
\end{corollary}

Observe that, if convexity at $\mu$ were replaced with convexity of $I_K$ on $\mathbb{P}(\Omega)$, then in view of Proposition \ref{prop:ConvexCPDEqual2}, this would be equivalent to the conditional positive definiteness of $K$. Thus,  Corollary \ref{cor:LocGlob3} recovers and strengthens part \eqref{lg3} of  Proposition \ref{prop:LocalMinGlobalMin}.  

The second corollary will provide a crucial implication in Theorem \ref{thm:CPDEquivalences}.
\begin{corollary}\label{cor:ConvMuGlobMin}
Let $K$ be a kernel on $\Omega \times \Omega$ and let $\mu \in \mathbb{P}(\Omega)$ be a $K$-invariant measure.  If $I_K$ is convex at $\mu$, then $\mu$ is a global minimizer of $I_K$ over $\mathbb{P}(\Omega)$. 
\end{corollary}

\begin{proof}[Proof of Theorem \ref{thm:CPDPotConstMin}]
Observe first that the constant $M$ is necessarily equal to $I_K (\mu)$:
$$ I_K (\mu) =  \int\limits_{\Omega} U_K^\mu(x)\, d \mu(x) =  \int\limits_{\operatorname{supp} (\mu) } M d \mu(x) = M.$$
For any $\nu \in \mathbb{P}(\Omega)$, we have that
\begin{equation*}
I_K(\mu, \nu) = \int\limits_{\Omega} U_K^\mu(x)\, d \nu(x) \geq % \int\limits_{\Omega} M d \nu(x) = M = \int\limits_{\Omega} U_K^\mu(x) d \mu(x) = 
I_K(\mu).
\end{equation*}
Convexity of  $I_K$   at $\mu$,  according to Lemma \ref{lem:ConvexEqual2}, is equivalent to the arithmetic mean inequality \eqref{eq:AM2}. Thus,
\begin{equation*}
I_K (\mu) \leq I_K(\mu, \nu) \leq \frac{1}{2} I_K(\mu) + \frac{1}{2} I_K(\nu),
\end{equation*}
so $I_K(\nu) \geq  I_K(\mu)$.
\end{proof}

\section{Invariant Measures and Minimizers with Full Support}\label{sec:Energy Basics}

It is now time to summarize the results of the previous sections. It may not be yet obvious, but we have proven (sometimes quite surprising) equivalences between many different notions.
We shall restrict our attention to the case when the measure $\mu$ is $K$-invariant (i.e. has constant potential) and has full support. As we have discussed before, these conditions are satisfied by many natural candidates (the uniform measure on the sphere or other two-point homogeneous spaces, the Haar measure on a compact topological group, etc). Though a majority of the implications are valid even just for $K$-invariant measures without the full support assumption, assuming that $\mu$ has full support truly ties the picture together. We shall carefully trace which of the conclusions require this condition.

We start with the following long list of equivalences.

\begin{theorem}\label{thm:CPDEquivalences}
Let  $K$ be a kernel on $\Omega \times \Omega$. Assume that there exists a measure  $\mu \in \mathbb{P}(\Omega)$, which is $K$-invariant and has full support, i.e.  $U_K^{\mu}(x) = I_K(\mu)$ for all $x\in \Omega$ and $\operatorname{supp}(\mu) = \Omega$. 

Then the following conditions are equivalent:

\begin{enumerate}

\item \label{P1} $K$ positive definite modulo a constant.

\item \label{P2} $K$ is conditionally positive definite.

\item \label{P3} $\mu$ is a local minimizer of $I_K$.

\item \label{P4} $\mu$ is a global minimizer of $I_K$ over $\mathbb{P}(\Omega)$.

\item \label{P5} $\mu$ is a global minimizer of $I_K$ over $\widetilde{\mathbb{P}}(\Omega)$.

\item \label{P6} $I_K$ is convex on $\mathbb P (\Omega)$ (or, equivalently, on $\widetilde{\mathbb{P}}(\Omega)$).

\item \label{P7} $I_K$ is convex at $\mu$.

\item \label{P8} The arithmetic mean inequality \eqref{eq:2inputAInequality} holds for all $\mu_1, \mu_2 \in \mathbb{P}(\Omega)$ (or, equivalently for all $\mu_1, \mu_2 \in \widetilde{\mathbb{P}}(\Omega)$).

\item \label{P9} The arithmetic mean inequality \eqref{eq:2inputAInequality} holds when $\mu_1 = \mu$.

\item \label{P10} The kernel $K$ can be represented as
$$ K(x,y) = \sum_{j = 1}^{\dim( L^2(\Omega, \mu))}
 \lambda_j \phi_j(x) \phi_j(y)$$
where the series converges uniformly and absolutely,    the function $\phi_1$ is constant, and $\lambda_j \geq 0$ for $j \geq 2$.

\end{enumerate}
\end{theorem}

\begin{figure}[h]
\centering
\begin{tikzpicture}
    \def\arr{Stealth[length=2.5mm,inset=0.5mm]}
    \node[circle, radius = .5, draw] (a) at (90:3)  {1};
    \node[circle, radius = .5, draw] (b) at (54:3)  {2};
    \node[circle, radius = .5, draw] (c) at (18:3)  {3};
    \node[circle, radius = .5, draw] (d) at (342:3) {4};
    \node[circle, radius = .5, draw] (e) at (306:3) {5};
    \node[circle, radius = .5, draw] (f) at (270:3) {6};
    \node[circle, radius = .5, draw] (g) at (234:3) {7};
    \node[circle, radius = .5, draw] (h) at (198:3) {8};
    \node[circle, radius = .5, draw] (i) at (162:3) {9};
    \node[circle, radius = .5, draw] (j) at (126:3) {10};
    \draw[arrows={-\arr}     , black , solid ,very thick] (a) --(b);
    \draw[arrows={\arr-\arr} , black , solid ,very thick] (b)--(f);
    \draw[arrows={\arr-\arr} , black , solid ,very thick] (b)--(h);
    \draw[arrows={-\arr}     , black , solid ,very thick] (d)--(c);
    \draw[arrows={-\arr}     , black , solid ,very thick] (e)--(d);
    \draw[arrows={-\arr}     , black , solid ,very thick] (f)--(g);
    \draw[arrows={\arr-\arr} , black , solid ,very thick] (f)--(h);
    \draw[arrows={\arr-\arr} , black , solid ,very thick] (g)--(i);
    \draw[arrows={-\arr}     , black , solid ,very thick] (h)--(i);
    \draw[arrows={-\arr}     , black , solid ,very thick] (j)--(a);
    \draw[arrows={-\arr}, black      , solid, decorate,decoration={snake,amplitude=.4mm,segment length=2mm,post length=3mm}] (b) to[out=134, in=10] (a);
    \draw[arrows={\arr-\arr}, black  , solid, decorate,decoration={snake,amplitude=.4mm,segment length=2mm,pre length=3mm,post length=3mm}] (b)--(e);
    \draw[arrows={-\arr}, black      , solid, decorate,decoration={snake,amplitude=.4mm,segment length=2mm,post length=3mm}] (c) to[out=298, in=62](d);
    \draw[arrows={-\arr}, black      , solid, decorate,decoration={snake,amplitude=.4mm,segment length=2mm,post length=3mm}] (g)--(d);
    \draw[arrows={-\arr}, black , dashed, thick ] (a) to [out=170 , in=46] (j);
    \draw[arrows={-\arr}, black , dashed, thick ] (d)--(a);
\end{tikzpicture}
\caption{Implications in  the proof of  Theorem \ref{thm:CPDEquivalences}: solid arrows are implications that hold without additional assumptions; wavy arrows require $K$-invariance, but not full support; dashed arrows represent the implications which do require the assumption of full support.}\label{fig1}
\end{figure}
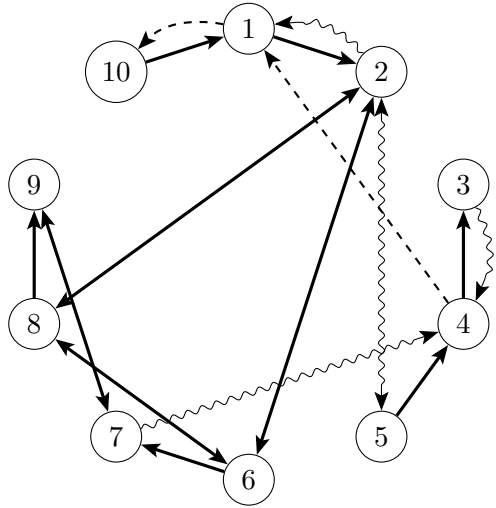

%\begin{figure}[t]
  %  \centering
   % \includegraphics[width=0.5\linewidth]{img/diagram.png}
   % \caption{Equivalences in Theorem \ref{thm:CPDEquivalences}: blue arrow are implications that hold without additional assumptions; green ones require $K$-invariance, but not full support; the orange arrows represent the  implications which do require the assumption of full support.}
 %   \label{fig1}
%\end{figure}

% \begin{figure}
% \includegraphics[scale=0.3]{img/diagram.png}
% \caption{Equivalences in Theorem \ref{thm:CPDEquivalences}: blue arrow are implications that hold without additional assumptions; green ones require $K$-invariance, but not full support; the orange arrows represent the  implications which do require the assumption of full support.}\label{fig1}
% \end{figure}

\begin{proof} For the reader's convenience the implications proving this theorem are summarized in Figure \ref{fig1}. 

We open with a list of  the implications that do not require any assumptions on $\mu$.   It is obvious that \eqref{P5} implies \eqref{P4}, which in turn implies \eqref{P3}. Also, \eqref{P6} implies \eqref{P7}, and similarly, \eqref{P8} implies \eqref{P9}.

The equivalence between \eqref{P2}, \eqref{P6}, and \eqref{P8} is proved in Lemmas \ref{lem:BHS2inputAInequality} and \ref{lem:ConvexEqual2} together with Proposition \ref{prop:ConvexCPDEqual2}.  Lemma \ref{lem:ConvexEqual2}  also establishes the equivalence between \eqref{P7} and \eqref{P9}.  Lemma \ref {lem:PD-CPD}   shows that  \eqref{P1} implies \eqref{P2}.

The following implications rely on the fact that $\mu$ is $K$-invariant, but do not require $\mu$ to have full support.  Lemma \ref{lem:CPDtoPDC} demonstrates that \eqref{P2} implies \eqref{P1}.  Theorem \ref{thm:CPDequivMIN} yields the equivalence of  \eqref{P2} and \eqref{P5}. Corollary \ref{cor:ConvMuGlobMin} %to Theorem \ref{thm:CPDPotConstMin} 
shows that \eqref{P4} follows from \eqref{P7}.  Finally, part \eqref{lg2} of Proposition \ref{prop:LocalMinGlobalMin} guarantees that \eqref{P3} implies \eqref{P4}. \\

The equivalence between  \eqref{P1} and  \eqref{P10} is discussed in Lemma \ref{lem:PDconstructed} and Mercer's Theorem (see Theorem \ref{thm:Mercer}), or more specifically Corollary \ref{cor:PDCharacterization}. Condition \eqref{P1} implies \eqref{P10} due to the fact that  the constant function $\mathbbm 1_\Omega$  is an eigenfunction of the Hilbert--Schmidt operator $T_{K,\mu}$ in $L^2 (\Omega, \mu)$ for any $K$-invariant measure $\mu$, and  the expansion in Mercer's Theorem is valid on all of $\Omega$, since $\operatorname{supp} (\mu) = \Omega$. The implication \eqref{P10} $\Rightarrow$ \eqref{P1} holds without any additional assumptions, according to Lemma \ref{lem:PDconstructed}.

In conclusion, we observe that  Lemma \ref{lem:PosDefonSupp}  demonstrates that \eqref{P4} implies \eqref{P1}, which closes the loop of implications  -- and, except for the standalone  equivalence between  \eqref{P1} and  \eqref{P10},  this is the only implication in our proof  where the fact that $\operatorname{supp}(\mu) = \Omega$ is used. Indeed, Lemma \ref{lem:PosDefonSupp}  only guarantees that the kernel $K$ is positive definite (up to constant) on the support of the minimizer.  Observe also  that  due to Theorem \ref{thm:Constant on Supp}, if \eqref{P4} holds and $\mu$ has full support, then $\mu$ is automatically $K$-invariant. 
\end{proof}

To reiterate, this theorem reveals several interesting novel  properties of  a $K$-invariant measure (with full support):

\begin{itemize}
\item Equivalence between minimization over the set  $\mathbb{P}(\Omega)$ of  probability measures  and the set $\widetilde{\mathbb{P}} (\Omega)$ of all signed measures of mass one.  This effect has been observed for rotationally invariant kernels on the sphere and the surface measure $\sigma$ by two of the authors and F.~Dai  \cite{BDM}. This is not necessarily the case in other settings. In particular, for the integral over the unit  ball ${\mathbb{B}^{d}}$  $$ \int\limits_{\mathbb{B}^{d}} \int\limits_{\mathbb{B}^{d}} \| x-y \|\, d\mu (x) d \mu (y),$$ according to \cite{Bj}, the unique maximizer over both  $\mathbb{P}(\mathbb{B}^{d})$ and $\mathbb{P}(\mathbb{S}^{d-1})$ is $\sigma$.  According to the aforementioned equivalence, $\sigma$ is also a maximizer over $\widetilde{\mathbb{P}}(\mathbb{S}^{d-1})$, while in the case of signed measures on the ball,  the maximizer does not exist \cite{HNW}. Observe, that one cannot expect such minimizers to exist in general, since $\widetilde{\mathbb{P}} (\Omega)$  is not weak$^*$ compact. 
\item Equivalence between  being a local and global energy minimizer.  This effect, in a slightly less general form, has been observed by the authors and their collaborators in \cites{BFGMPV, BGMPV}.
\item Equivalence between other local and global properties: e.g., the energy functional $I_K$ is convex on the whole set of probability measures  if and only if it is convex just at the special measure $\mu$.
\item Equivalence between conditional positive definiteness of the kernel and positive definiteness up to constant, which is not true in general. 
\item We also note that, due to  Corollary \ref{cor:LocMinEquiv}, any of the three notions of local minimizers (directional, in  total variation, and  in the Wasserstein distance $W_p$, $1\le p< \infty$,  for metric spaces $\Omega$) may be assumed in part \eqref{P3} of  Theorem \ref{thm:CPDEquivalences}, as they are equivalent for $K$-invariant measures. 
\end{itemize}

We now formulate similar theorems for kernels which have the ``strict''  version of the properties and for positive definite kernels. We start with the latter.

\begin{theorem}\label{thm:PDEquivalences}
Suppose that $K$ is a kernel on $\Omega \times \Omega$  and that there exists a measure  $\mu \in \mathbb{P}(\Omega)$, which is $K$-invariant and has full support, i.e.  $U_K^{\mu}(x) = I_K(\mu)$ for all $x\in \Omega$ and $\operatorname{supp}(\mu) = \Omega$. Then the following conditions  are equivalent:
\begin{enumerate}
\item \label{Q1} The kernel $K$ is positive definite.

\item \label{Q2} The geometric mean inequality \eqref{eq:2inputGInequality} and $I_K(\mu_1) \geq 0$ hold for all $\mu_1, \mu_2 \in \mathbb{P}(\Omega)$.

\item \label{Q3} The measure $\mu$ is a global minimizer of $I_K$  and satisfies $I_K(\mu) \geq 0$.

\item \label{Q3a} $I_K$ is convex on $\mathcal M (\Omega)$. 

\item \label{Q4} The kernel $K$ can be represented as
$$ K(x,y) = \sum_{j=1}^{\dim( L^2(\Omega, \mu))} \lambda_j \phi_j(x) \phi_j(y)$$
where the series converges uniformly and absolutely, and $\lambda_j \geq 0$ for $j \geq 1$.

\item \label{Q5} There exists some symmetric $k \in L^2(\Omega \times \Omega , \mu \times \mu)$ such that for all $x, y \in \Omega$,
$$ K(x,y) = \int\limits_{\Omega} k(x,z) k(z,y)\, d \mu(z).$$
\end{enumerate}

\end{theorem}

\begin{proof}
Lemma \ref{lem:BHS2inputGInequality}, Proposition \ref{p:mircea}, Corollary \ref{cor:PDCharacterization}, and Proposition \ref{prop:Convolution} show the equivalence of \eqref{Q1}, \eqref{Q2}, \eqref{Q3a}, \eqref{Q4}, and \eqref{Q5}. Positive definiteness, i.e.\ condition \eqref{Q1}, guarantees  that $I_K(\mu) \geq 0$, and that $\mu$ is a minimizer, due to Theorem \ref{thm:CPDEquivalences}. Conversely, Lemma \ref{lem:PosDefonSupp} shows that \eqref{Q3} implies \eqref{Q1}, finishing our proof. Observe also that according to Theorem  \ref{thm:CPDEquivalences} it does not matter whether we mean global minimization over ${\mathbb{P}} (\Omega)$ or $\widetilde{\mathbb{P}} (\Omega)$ in condition \eqref{Q3}. 
\end{proof}

\begin{theorem}
\label{thm:CSPDEquivalences}
Suppose that $K$ is a kernel on $\Omega \times \Omega$ and that there exists a measure  $\mu \in \mathbb{P}(\Omega)$ which is $K$-invariant, i.e.  $U_K^{\mu}(x) = I_K(\mu)$ for all $x\in \Omega$. Then the following conditions  are equivalent:
\begin{enumerate}
\item \label{R1} $K$ is conditionally strictly positive definite.

\item \label{R2} $K$ is strictly positive definite modulo a constant.

\item \label{R3} $\mu$ is the unique minimizer of $I_K$ over  $\widetilde{\mathbb{P}}(\Omega)$.

\item \label{RR} $I_K$ is strictly convex on $ \mathbb P(\Omega) $ (or, equivalently, on $\widetilde{\mathbb{P}}(\Omega)$).

\end{enumerate}

\noindent If in addition $\operatorname{supp}(\mu) = \Omega$, i.e. $\mu$ has full support, then each of the conditions \eqref{R1}--\eqref{R3} implies the following

\begin{enumerate}\setcounter{enumi}{4}
\item \label{R4} The kernel $K$ can be represented as
$$ K(x,y) = \sum_{j=1}^{\dim( L^2(\Omega, \mu))} \lambda_j \phi_j(x) \phi_j(y)$$
where $\{\phi_j\} $ is the orthonormal basis  consisting of eigenfunctions of the Hilbert--Schmidt operator $T_{K,\mu}$ in $L^2 (\Omega, \mu)$, the function $\phi_1$ is a constant, the series converges uniformly and absolutely, and $\lambda_j > 0$ for $j \geq 2$.

\end{enumerate}

\noindent Moreover, if the span of  $\{ \phi_j \}$ is dense in $C(\Omega)$, then \eqref{R4} also implies conditions \eqref{R1}--\eqref{R3}. % In particular, in this case,  condition \eqref{R4} also implies that 
\end{theorem}

\begin{proof}
    Lemma \ref{lem:PD-CPD} shows that \eqref{R2} implies \eqref{R1}, while Lemma \ref{lem:CPDtoPDC} provides the converse implication.  The equivalence of \eqref{R3} and \eqref{R1} is proved in Theorem \ref{thm:CPDequivMIN}. For equivalence of \eqref{R1} and \eqref{RR}, see the discussion after Proposition~\ref{prop:ConvexCPDEqual2}.

Before we turn to dealing with condition \eqref{R4}, recall that $K$-invariance   of $\mu$ implies that a constant  is an eigenfunction of  the operator $T_{K,\mu}$, so we shall assume that $\phi_1 = \mathbbm{1}_{\Omega}$. 

Now we show that \eqref{R3} implies \eqref{R4}. Since $\mu$ minimizes  $I_K$ and has full support, by Theorem \ref{thm:CPDEquivalences}, $K$ is positive definite up to a constant, and since $\phi_1 =1$,  by Mercer's Theorem, the expansion in part \eqref{R4} holds with some $\lambda_j \in \mathbb{R}$, such that $\lambda_j \ge 0$ for $j\geq 2$,  and with the series converging uniformly and absolutely.  Suppose, indirectly, that $\mu$ is the unique minimizer of $I_K$ over $\widetilde{\mathbb{P}}(\Omega)$, but there exists some $l \geq 2$ such that $\lambda_l = 0$. Let $d \nu(x) = (1 + \phi_l(x))\, d\mu(x)$. Orthogonality implies that  $\int\limits_\Omega \phi_l (x) \, d\mu(x) = 0$, therefore $\nu \in \widetilde{\mathbb{P}}(\Omega)$.
 Then we obtain
\begin{align*}
I_K(\nu) & = \int\limits_{\Omega} \int\limits_{\Omega} K(x,y) (1 + \phi_l(x)) (1 + \phi_l(y))\, d \mu(x)d\mu(y) \\
& = I_K(\mu) +  2  \langle T_{K,\mu} \phi_l, \mathbbm{1}_{\Omega} \rangle_{L^2(\Omega, \mu)} +  \langle T_{K,\mu} \phi_l, \phi_l  \rangle_{L^2(\Omega, \mu)}\\
& = I_K (\mu) + 2 \lambda_l  \langle   \phi_l, \mathbbm{1}_{\Omega} \rangle_{L^2(\Omega, \mu)} + \lambda_l  \|  \phi_l \|_{L^2(\Omega, \mu)}^2  = I_K(\mu), 
%& = I_K (\mu) + \lambda_l   = I_K(\mu),
\end{align*}
which contradicts the fact that $\mu$ is the unique minimizer over $\widetilde{\mathbb{P}}(\Omega)$. 

Finally, we show that \eqref{R4} implies \eqref{R2}  under the aforementioned additional assumption. Let $K'(x,y) = K(x,y) - \lambda_1 +1$ and $\nu \in \mathcal{M}(\Omega)$.  Then
\begin{align*}
I_{K'}(\nu) & = \int\limits_{\Omega} \int\limits_{\Omega} K'(x,y)\, d \nu(x) d\nu(y) \\
& = (\nu(\Omega))^2 +  \sum_{j =2}^{\dim( L^2( \Omega, \mu))} \int\limits_{\Omega} \int\limits_{\Omega} \lambda_j \phi_j(x) \phi_j(y)\ d \nu(x) d\nu(y) \\
& = (\nu(\Omega))^2 +  \sum_{j =2}^{\dim( L^2( \Omega, \mu))} \lambda_j  \Big( \int\limits_{\Omega} \phi_j(x)\, d \nu(x) \Big)^2 \geq 0.
\end{align*}
Clearly, the only way that $I_{K'}(\nu) =0$ is if $\int\limits_{\Omega} \phi_j(x)\, d \nu(x) = 0$ for all $j \geq 1$. By the density of $\operatorname{span} \{ \phi_j \}_{j\ge 1} $ in $C(\Omega)$, we can conclude that this implies $\nu = 0$, so $K'$ must be strictly positive definite, which completes the  proof.
\end{proof}

We conclude with the  remark that the additional condition imposed for the sufficiency of condition \eqref{R4} is not very restrictive in practice. For example in the case of rotationally invariant kernels on the sphere and $\mu = \sigma$, the eigenfunctions $\phi_j$ are simply spherical harmonics, which span all polynomials on the sphere and thus their span is dense in the space continuous functions.

%%%%%%%%%%%%%%%%%%%%%%%%%%%%%%%%%%%%%%%%%%%%%%%%%
%%%%%%%%%%%%%%%%%%%%%%%%%%%%%%%%%%%%%%%%%%%%%%%%%
%%%%%%%%%%%%%%%%%%%%%%%%%%%%%%%%%%%%%%%%%%%%%%%%%
%%%%%%%%%%%%%%%%%%%%%%%%%%%%%%%%%%%%%%%%%%%%%%%%%

\section{The Generalized Stolarsky Principle on Compact  Spaces}\label{sec:MetricStol}

Discrepancy theory analyzes discrete point configurations by comparing their distribution to some fixed (usually uniform) measure $\mu$ on a class of test sets. More precisely, the local discrepancy of  a finite configuration $\omega_N = \{z_1,\dots, z_N\} \subset \Omega$ with respect to  a set $A \subseteq \Omega$ is defined as 
\begin{equation}\label{eq:discA}
D (\omega_N, A) =   \frac{1}{N} \sum_{i=1}^N \mathbbm{1}_A ( z_i)  - \mu (A) .
\end{equation}
One then takes the supremum over some class $\mathcal A$ of test sets $A$  (extremal discrepancy) or a certain average, if the class $\mathcal A$ admits a natural measure (quadratic averages lead to the so-called $L^2$  discrepancy). These quantities provide important information about the distribution of the discrete set $\omega_N$  \cite{BC, Ma, KN}.

Discrepancy is closely related to discrete energy  \begin{equation}\label{eq:2DiscreteEnergy}
E_K(\omega_N) :=  \frac{1}{N^2} \sum_{x,y \in \omega_N}  K (x,y). 
\end{equation}
The definitions of discrete \eqref{eq:2DiscreteEnergy} and continuous \eqref{eq:2ContEnergy} energies are compatible in the sense that 
\begin{equation}
E_{K}(\omega_N) = I_{K}(\mu_{\omega_N}), \; \; \; \; \text{ where } \mu_{\omega_N} = \frac{1}{N} \sum_{x \in \omega_N} \delta_{x}.
\end{equation}
Similarly, the definition of discrepancy may be extended to the continuous setting by replacing the empirical measure  $\mu_{\omega_N} (A) = \frac{1}{N} \sum_{i=1}^N \mathbbm{1}_A ( z_i) $ by an arbitrary measure $\nu \in \widetilde{\mathbb P} (\Omega)$.  \\

The foundational example of the connection between energy and discrepancy is given by the classical Stolarsky Invariance Principle \cite{St} which states that 
\begin{equation}\label{eq:StolOrig}
c_d   D^{\, 2}_{L^2,\textup{cap}} (\omega_N)   =   \int\limits_{\mathbb{S}^{d-1}} \int\limits_{\mathbb{S}^{d-1}} ||  x- y || \, d\sigma (x)\, d\sigma (y)\,\,  - \,\, \frac{1}{N^2} \sum_{i,j = 1}^N || z_i - z_j ||,   
\end{equation}
where $D_{L^2,\textup{cap}} (\omega_N)$ is the $L^2$ spherical cap discrepancy 
\begin{equation}\label{eq:DefSphCapDisc}
D^{\, 2}_{L^2,\textup{cap}} (\omega_N)   =  \int\limits_{-1}^1 \,  \int\limits_{\mathbb{S}^{d-1}}  \bigg| \frac1{N} \sum_{j=1}^N \mathbbm{1}_{C(x,h)} (z_j) -  \sigma \big( C(x,h) \big) \bigg|^2 d\sigma (x) \, dh
\end{equation}
and $C(x,h)$ denotes a spherical cap of height $h\in [-1,1]$ centered at $x\in \mathbb S^{d-1}$, i.e. $C(x,h ) = \{ z \in \mathbb{S}^{d-1}:\, \langle z, x \rangle > h \}$. This principle shows that minimizing the $L^2$ spherical cap discrepancy is equivalent to an energy optimization problem -- maximizing the pairwise sum of Euclidean distances. Observe that by setting $K (x,y) = 1 - c_d^{-1} \| x-y \|$, one can rewrite \eqref{eq:StolOrig} as 
\begin{equation}
 D^{\, 2}_{L^2,\textup{cap}} (\omega_N) = I_K (\nu) - I_K (\sigma)
\end{equation}
with $\nu = \mu_{\omega_N}$, and the right-hand side of this relation is familiar to the reader from  identity \eqref{eq:Linear1} of Lemma \ref{lem:NuMinusMu}. 

In the recent years, numerous authors (including the first two authors of this paper) revisited this fascinating fact, extended it, and applied it to various problems of discrete geometry and optimization: new proofs of the original Stolarsky Invariance Principle have been given in \cite{BrD, BDM, HBZO}, it has been extended to geodesic distances and other rotationally invariant kernels on the sphere  \cite{BD, BDM},
to   compact, connected, two-point homogeneous spaces \cite{Sk1, Sk2} and to the Hamming cube \cite{Ba, BS}, and applied to two problems of Fejes T\'oth on sums of various distances on the  sphere and in  projective spaces \cite{BDM, BM}. \\

In this paper, we prove a general version of this principle on arbitrary compact spaces which extends some of the versions mentioned above and does not use any structural information about the underlying domain. Let $\Omega$ be a compact  topological  space and let us  fix a measure $\mu \in \mathbb{P}( \Omega)$ -- this will usually be an energy minimizing (equilibrium) measure or an invariant measure of full support (its role is similar to that of $\sigma$ in the spherical case). We now  define the \textit{$L^2$ discrepancy of an arbitrary  probability  measure $\nu \in \mathbb{P}(\Omega) $  (or even a signed measure $\nu \in \widetilde{\mathbb{P}} (\Omega)$) relative to the  measure $\mu$ with respect to the function $k: \Omega \times \Omega \rightarrow \mathbb R$} by the identity
\begin{equation}\label{eq:dl2mu}
\begin{aligned}
D_{L^2, k, \mu}^{\, 2} (\nu) &= \int\limits_\Omega \bigg|  \int\limits_\Omega k(x,y)\, d\nu(y) - \int\limits_\Omega k(x,y)\, d\mu(y)  \bigg|^2 d\mu(x)\\
&= \int\limits_\Omega \bigg|  \int\limits_\Omega k(x,y)\, d\big( \nu -  \mu \big) (y)    \bigg|^2 d\mu(x).
\end{aligned}
\end{equation}

\noindent When $\nu$ is the equal-weight discrete measure associated to the $N$-point set $\omega_N= \{ z_1, ..., z_N\} \subset \Omega$, i.e. $\displaystyle{ \nu = \frac1{N} \sum_{i=1}^N \delta_{z_i}}$, this  becomes the discrepancy of the set $\omega_N$ with respect to $k$:
\begin{equation}\label{eq:dl2z}
D_{L^2, k, \mu}^{\, 2} (\omega_N) = D_{L^2, k, \mu}^{\, 2} \Big(\frac1{N} \sum_{i=1}^N \delta_{z_i} \Big) =  \int\limits_\Omega \bigg|   \frac1{N}  \sum_{i=1}^N   k (x, z_i)  - \int\limits_\Omega k(x,y)\, d\mu(y)  \bigg|^2 d\mu(x)
\end{equation}
Notice that, in the spherical case, setting $k(x,y) = \mathbbm{1}_{\{\langle x,y  \rangle > h \} } = \mathbbm{1}_{C(x,h)} (y) $, one finds that \eqref{eq:dl2z} is equal to  the inner integral in \eqref{eq:DefSphCapDisc} with the integrand of the form \eqref{eq:discA}. Therefore, this definition is indeed an extension of the classical notion of discrepancy with arbitrary functions $k$ in place of indicators of test sets. 
%Observe that changing the function $k$ by an additive  constant  does not change the value of the discrepancy.

We can now obtain the following  general version of the Stolarsky Invariance Principle:

\begin{theorem}[Generalized Stolarsky Principle]\label{thm:genStol}
Let $K$ be a positive definite (modulo an additive constant $C$) kernel on $\Omega\times \Omega$. Let us assume that $\mu \in \mathbb{P}(\Omega)$ is a $K$-invariant measure  with full support, i.e.  $U_K^{\mu}(x) = I_K (\mu) $ for  all  $x\in \Omega$ and $\operatorname{supp} (\mu) = \Omega$. Then for every measure $\nu \in \widetilde{\mathbb{P}} (\Omega)$, we have the following identity.
\begin{equation}\label{eq:stol1Metric}
I_K (\nu) - I_K (\mu) = D_{L^2, k, \mu}^{\, 2} (\nu),
\end{equation}
where the function $k \in L^2(\Omega \times \Omega , \mu \times \mu)$ is  as in part \eqref{Q5}  of  Theorem    \ref{thm:PDEquivalences} applied to the positive definite kernel $K+C$. 

In particular, for a discrete set $\omega_N= \{z_1, ..., z_N \} \subset \Omega$, 
\begin{equation}\label{eq:stol2Metric}
E_K (\omega_N ) - I_K (\mu) = D_{L^2, k, \mu}^{\, 2} (\omega_N).
\end{equation}

\end{theorem}

This theorem has the following immediate corollary:
\begin{corollary}\label{cor:genStol}
Let $K$ be a   kernel on $\Omega\times \Omega$. Assume that $\mu \in \mathbb{P}(\Omega)$ is a global minimizer of the energy functional $I_K$ over $\mathbb{P}(\Omega)$ with $\operatorname{supp} (\mu ) = \widetilde{\Omega} \subseteq \Omega$.  Then identity \eqref{eq:stol1Metric} holds for any signed measure $\nu$ with total mass one, whose support is contained in the support of $\mu$, i.e. $\nu \in \widetilde{\mathbb{P}} (\widetilde{\Omega})$. Similarly, relation \eqref{eq:stol2Metric} holds for any point set $\omega_N = \{z_1, ..., z_N \} \subset \widetilde\Omega$.
\end{corollary}
\begin{proof} If $\mu$ is a global minimizer of $I_K$, by Theorem \ref{thm:Constant on Supp}, the potential $U_K^\mu$ is constant on $\widetilde\Omega$, i.e. $\mu$ is $K$-invariant and has full support if viewed as an element of $\mathbb{P}(\widetilde{\Omega})$. Moreover,  according to Lemma \ref{lem:PosDefonSupp}, the kernel $K$ is positive definite (modulo a constant) on $\widetilde{\Omega}$. Therefore, the statement follows directly from Theorem \ref{thm:genStol} applied to $\widetilde\Omega$ in place of $\Omega$. 
\end{proof}

We now turn to the proof of the generalized Stolarsky principle: 

\begin{proof}[Proof of Theorem \ref{thm:genStol}] 
Without loss of generality, we can assume that $K$ is positive definite, since adding a constant to $K$  affects neither the invariance of $\mu$ nor the difference $I_K (\nu) - I_K ({\mu})$. We  can now use  the crucial identity \eqref{eq:Linear1} of Lemma \ref{lem:NuMinusMu},  as well as part \eqref{Q5}  of  Theorem    \ref{thm:PDEquivalences}, to obtain
\begin{align}
\nonumber I_K (\nu) - I_K (\mu) & = I_K (\nu - \mu) =  \int\limits_{\Omega} \int\limits_{\Omega}   K(x,y)  d\big(\nu - \mu \big) (x)  d\big(\nu - \mu \big) (y)\\
\label{eq:middle} & =   \int\limits_{\Omega} \int\limits_{\Omega}   \int\limits_\Omega k(x,z) k (z,y) \, d\mu (z)   d\big(\nu - \mu \big) (x)  d\big(\nu - \mu \big) (y) \\
\nonumber & = \int\limits_\Omega \bigg|   \int\limits_\Omega k(x,z)  d\big(\nu - \mu \big) (x) \bigg|^2  \, d\mu (z)  = D_{L^2, k, \mu}^{\, 2} (\nu).
\end{align}
\end{proof}

{\emph{Remark:}} Observe that, for $k \in L^2(\Omega \times \Omega , \mu \times \mu)$, it is not technically obvious that the definition of $D_{L^2, k, \mu}^2 (\nu)$ in \eqref{eq:dl2mu} is properly justified: we do not know a priori that $k$ is integrable with respect to $\nu$, only with respect to $\mu$. (This problem does not occur in the discrete case since we know that $k(\cdot, z_i) \in L^2 (\Omega, \mu )$ for each $i=1,...,N$.)  However, the proof of Stolarsky principle \eqref{eq:stol1Metric} demonstrates that  the $L^2$ discrepancy $D_{L^2, k, \mu}^2 $ is well-defined for any Borel measure $\nu$. Indeed, the inner integral with respect to $d\mu (z)$ in \eqref{eq:middle} is defined according to part \eqref{Q5}  of Theorem    \ref{thm:PDEquivalences}, and, moreover, produces the  function $K(x,y)$, which is continuous and therefore integrable with respect to  the finite Borel measure $(\nu - \mu) \times (\nu - \mu)$ on $\Omega \times \Omega$. Hence Fubini's theorem applies and $$\int\limits_{\Omega}  \int\limits_{\Omega} k(x,z) k(y,z)  d\big(\nu - \mu \big) (x)   d\big(\nu - \mu \big) (y) =   \bigg| \int\limits_{\Omega} k (x,z)   d\big(\nu - \mu \big) (x) \bigg|^2 $$ is finite for $\mu$-a.e. $z$ and is integrable with respect to $d\mu (z)$, i.e. $D_{L^2, k, \mu}^2 (\nu)$ is well-defined.

\section{The Spherical Case}\label{sec:Sphere}

We end with a brief discussion of  the results of  prior sections in the particular case when the domain is the sphere in  Euclidean space. More specifically,  let $\Omega = \mathbb S^{d-1}$ and let  $K$ be a rotationally invariant kernel  on the sphere, i.e. $K (x,y ) = F (\langle x,y \rangle)$ for each $x,y \in \mathbb S^{d-1}$, where $F \in C[-1,1]$.  By a slight abuse of notation, we shall also call the kernel $F$ and  the energy $I_F$. 

 We immediately observe that $\sigma$, the normalized surface measure on $\mathbb S^{d-1}$, is an $F$-invariant measure with full support. Therefore, all of the results of Sections \ref{sec:PosDefKer}--\ref{sec:Energy Basics} apply. In particular, Theorem  \ref{thm:CPDEquivalences} holds with $\mu = \sigma$. Some of its conclusions are interesting even in this classical case. We  do not completely restate Theorems \ref{thm:CPDEquivalences} and \ref{thm:PDEquivalences} for $\Omega = \mathbb S^{d-1}$, but simply summarize  some interesting facts.\\

For any function $F\in C[-1,1]$ which generates a kernel $F (\langle x,y \rangle)$ on $\mathbb S^{d-1} \times \mathbb S^{d-1}$,

\begin{enumerate}[(i)]
\item\label{Sp1} conditional positive definiteness of $F$ on the sphere  is equivalent to positive definiteness up to an additive constant; 
\item\label{Sp2} the facts that $\sigma$ minimizes in three different ways $I_F$ (locally, globally over probability measures, globally over signed measures of total mass one) are all equivalent to each other;
\item\label{Sp3} in turn, the fact that $\sigma$ minimizes $I_F$ is equivalent to conditional positive definiteness of $F$ on the sphere.
\end{enumerate}
Some  results of part \eqref{Sp2} have been observed in \cite{BDM, BGMPV} and part \eqref{Sp3} is well known, see, e.g., \cite{BDM} (but heuristically it goes back to  \cite{Sc}). % , whatelse?}. 

Further considerations are connected to the Gegenbauer expansions of positive definite kernels. The Gegenbauer polynomials $\{ C_n^\lambda\}_{n\ge 0} $ form an orthogonal basis of the space $ L^2 ([-1,1], w_\lambda)$  of  functions on $[-1,1]$, which are square integrable with respect to the weight $w_\lambda (t)   = (1-t^2)^{\lambda - \frac12}$.  These functions are closely related to harmonic analysis on the sphere, see e.g. \cite{DX} for details, when  $\lambda = \frac{d-2}{2}$ (which we assume for the rest of the section).

The kernel $F$ can be expanded in Gegenbauer series 
\begin{equation}\label{eq:GegenExpansion}
F(t) =  \sum_{n=0}^\infty \widehat{F}(n, \lambda) \frac{n+\lambda}{\lambda} C_n^\lambda (t)  %\quad t \in [-1,1]
\end{equation}
in the sense of  $ L^2 ([-1,1], w_\lambda)$.  In addition, the Funk--Hecke formula, which states that
\begin{equation}\label{eq:Funk-Hecke}
\int\limits_{\mathbb{S}^{d-1}} F(\langle x , y \rangle) Y_n(y)\, d\sigma(y) = \widehat{F}(n, \lambda) Y_n(x)
\end{equation} 
whenever $Y_n$ is a spherical harmonic of degree $n$ on $\mathbb S^{d-1}$, demonstrates that  the spherical harmonics are exactly the eigenfunctions of the Hilbert--Schmidt operator $T_{F,\sigma}$, with the corresponding Gegenbauer coefficients as eigenvalues. Thus, equivalence between positivity of $T_{F,\sigma}$ and positive definiteness of $F$, recovers the  seminal result of Schoenberg \cite{Sc} which asserts that positive definiteness on the sphere  is equivalent to non-negativity of Gegenbauer coefficients.  Moreover, if $\{ Y_n,k \}_{k=1}^{ \dim(\mathcal{H}_n^d)}$ is an orthonormal basis of the space $ \mathcal{H}_n^d$ of spherical harmonics of degree $n$ on $\mathbb S^{d-1}$, then the {\it{addition formula }}
\begin{equation}\label{eq:addSphHarm}
    \sum_{j=1}^{ \dim(\mathcal{H}_n^d)} Y_{n,j}(x) Y_{n,j}(y) = \frac{n+\lambda}{\lambda} C_n^\lambda(\langle x,  y\rangle)\   \ \textup{ for all }\,\, x,y\in \mathbb{S}^{d-1},
\end{equation}
together with Mercer's Theorem, Theorem \ref{thm:Mercer}, implies that for positive definite $F$ the following series converges absolutely and uniformly
\begin{align*}
F(\langle x, y\rangle ) = \sum_{n=0}^\infty \widehat{F} (n,\lambda) \sum_{k=1}^{\dim(\mathcal{H}_n^d)} Y_{n,k} (x) Y_{n,k} (y)  = \sum_{n=0}^\infty \widehat{F} (n,\lambda) \frac{n+\lambda}{\lambda} C^\lambda_n (\langle x, y\rangle ),
\end{align*}
and hence the Gegenbauer expansion of a function $F \in C[-1,1]$, which is  positive definite on the sphere,  must be  absolutely and uniformly convergent. Absolute convergence of the Gegenbauer expansions of positive definite functions has been observed before \cite{G, BD} using special properties of Gegenbauer polynomials -- here we see that this fact is a consequence of general spectral theory.  Part \eqref{Q4} of Theorem \ref{thm:PDEquivalences} and part \eqref{P10}  of Theorem \ref{thm:CPDEquivalences} may be restated as follows:
\begin{enumerate}[(i)]\setcounter{enumi}{3}
\item\label{Sp4} The kernel $F$ is (conditionally) positive definite on the sphere if and only if the Gegenbauer coefficients satisfy $\widehat{F} (n,\lambda) \ge 0 $ for all $n\ge 0$ ($n\ge 1$). Moreover, in this case, the Gegenbauer expansion \eqref{eq:GegenExpansion} converges uniformly and absolutely. 
\end{enumerate}
Similarly, one obtains the characterization in terms of the ``convolution square root'', as in part \eqref{Q5} of Theorem \ref{thm:PDEquivalences}:
\begin{enumerate}[(i)]\setcounter{enumi}{4}
\item\label{Sp5} positive definiteness of $F$ on the sphere is equivalent to the existence  of the function $f \in L^2 ([-1,1], w_\lambda)$    such that 
\begin{equation}\label{eq:3-1}
    F(\langle x,y \rangle)=\int\limits_{\mathbb{S}^{d-1}} f(\langle x, z \rangle) f(\langle z, y\rangle )\, d\sigma(z),\   \ x, y\in \mathbb{S}^{d-1}. 
\end{equation}
\end{enumerate}
This equivalence had been stated in \cite{BD} and the construction  of the function $f$ from the kernel $F$ is almost identical to that of the function $k$ from the kernel $K$ in Proposition \ref{prop:Convolution}: one chooses the Gegenbauer coefficients of $f$ so that $( \widehat{f} (n,\lambda) )^2 = \widehat{F} (n,\lambda)$, which is mimicked in \eqref{eq:k}. In addition, \eqref{eq:3-1} was used to prove a general Stolarsky Principle on the sphere \cite{BDM, BM}: for each $\nu \in \mathbb P (\mathbb S^{d-1})$, 
\begin{equation}
I_F (\nu) - I_F (\sigma) = D^{\, 2}_{L^2, f,\sigma } (\nu),
\end{equation}
which we have generalized to arbitrary compact domains in Theorem  \ref{thm:genStol}.\\

We finish by mentioning that our results similarly apply to the setting of two-point homogeneous compact spaces, both connected (e.g., projective spaces)  and discrete (e.g., Hamming cube) with the corresponding uniform measure, as well as compact topological groups with the Haar measure, although we do not pursue these directions in this paper.

\end{document}